\documentclass{amsart}
\usepackage{amssymb}
\usepackage{hyperref,color,xcolor}
\definecolor{wong_orange}{RGB}{230, 159, 0}
\definecolor{wong_skyblue}{RGB}{86, 180, 233}
\definecolor{wong_bluishgreen}{RGB}{0, 158, 115}
\definecolor{wong_yellow}{RGB}{240, 228, 66}
\definecolor{wong_blue}{RGB}{0, 114, 178}
\definecolor{wong_vermillion}{RGB}{213, 194, 0}
\definecolor{wong_reddishpurple}{RGB}{204, 121, 167}

\usepackage{tikz}
\usetikzlibrary{arrows,graphs,graphs.standard}

\theoremstyle{plain}
\newtheorem{thm}{Theorem}[section]
\newtheorem{prop}[thm]{Proposition}
\newtheorem{lem}[thm]{Lemma}
\newtheorem{cor}[thm]{Corollary}

\theoremstyle{definition}
\newtheorem{defn}[thm]{Definition}

\theoremstyle{remark}

\numberwithin{equation}{section}

\newcommand{\N}{\mathbb{N}}
\newcommand{\seq}[1]{\langle #1 \rangle}

\newcommand{\leW}{\leq_{\mathrm{W}}}

\newcommand{\lecW}{\leq^\ast_{\mathrm{W}}}
\newcommand{\sleW}{<_{\mathrm{W}}}
\newcommand{\lesW}{\leq_{\mathrm{sW}}}
\newcommand{\lecsW}{\leq^\ast_{\mathrm{sW}}}

\newcommand{\eqW}{\equiv_{\mathrm{W}}}
\newcommand{\eqsW}{\equiv_{\mathrm{sW}}}
\newcommand{\wphp}[3]{(#2\not\hookrightarrow #3)#1}

\newcommand{\C}{\mathsf{C}}

\newcommand{\AoUC}{\mathsf{AUC}}

\newcommand{\mflim}{\mathsf{lim}}
\newcommand{\id}{\mathsf{id}}
\newcommand{\ACC}{\mathsf{ACC}}

\newcommand{\mfP}{\mathsf{P}}
\newcommand{\mfQ}{\mathsf{Q}}

\newcommand{\gred}{\leq}

\newcommand{\gequiv}{\equiv}

\begin{document}

\title{Finite combinatorics and computability theory}


\author[Dzhafarov]{Damir D. Dzhafarov}
\address{Department of Mathematics\\University of Connecticut\\Storrs CT, USA 06107}
\email{damir@math.uconn.edu}

\author[Goh]{Jun Le Goh}
\address{Department of Mathematics\\National University of Singapore\\Singapore 119076}
\email{gohjunle@nus.edu.sg}

\date{\today}

\begin{abstract}
	We prove that the existence of finite combinatorial objects such as affine planes, mutually orthogonal Latin squares, and resolvable balanced incomplete block designs can be reformulated as the existence of certain algorithmic reductions between problems related to the pigeonhole principle. We then study the latter using counting arguments and computability theory. In particular, we demonstrate that computability theoretic techniques can be used to refine and prove new results in finite combinatorics.
\end{abstract}

\thanks{The authors thank the Institute for Mathematical Sciences at the National University of Singapore for hosting them in the summer of 2023, during which time this project was begun. The authors thank David Belanger for helpful discussions. Dzhafarov was partially supported by grants DMS 2452079 and DMS 1854355 from the National Science Foundation of the United States. Goh thanks Hao Huang and Jeck Lim for helpful conversations.}

\maketitle

\section{Introduction}\label{sec:intro}

Notions of reduction are widespread in mathematics and computer science. In computational complexity theory, say, one considers polynomial-time reductions between decision problems, or counting problems involving finite objects. In computability theory, one instead studies reductions defined in terms of Turing reducibility. Informally, a function $f : \mathbb{N} \to \mathbb{N}$ is \emph{Turing reducible} to (or \emph{computable from}) a function $g : \mathbb{N} \to \mathbb{N}$, written $f \leq_{\rm T} g$, if there is an algorithm for computing $f(n)$ from $n$ that is allowed to query the graph of $g$ during the computation process (in parlance, the algorithm uses $g$ as an \emph{oracle}). The definition can be extended from functions to subsets of $\mathbb{N}$, by identifying these with their characteristic functions. A function (or set) is \emph{computable} if it is Turing reducible to $\emptyset$. Importantly, the definition of Turing reducibility is not concerned with time or space complexity, which makes all finite sets computable, regardless of size. Reductions defined in terms of Turing reducibility are thus most often studied between problems that involve infinite sets.

As an example, consider the problem of finding an infinite monotone sequence in a given countable linear order. We can reduce this problem to that of finding an infinite homogeneous set for a given coloring of the unordered pairs of integers (i.e., to Ramsey's theorem for pairs). Indeed, suppose we are given a countable linear order; say its domain is $\mathbb{N}$, and the ordering relation is $\leq_L$. Define a coloring $c$ of the unordered pairs of integers as follows: for $x <y$ in $\mathbb{N}$, set $c(\{x,y\})$ to be $0$ if $x <_L y$, and $1$ if $x >_L y$. Now it is easily seen that if $H$ is any infinite homogeneous set for $c$ (i.e., if $c$ is constant on the unordered pairs of elements of $H$) then $H$ is monotone for $L$.

In the above example, (the graph of) $c$ is not only computable from (the graph of) $\leq_L$, but it is \emph{uniformly} computable. That is, the algorithm for converting $\leq_L$ to $c$ does not depend on $\leq_L$. Of course, each output of the algorithm (color of a given unordered pair) depends on the answers to the queries the algorithm makes of the oracle (how certain elements are ordered). But which queries should be made, and what the algorithm should do with the possible answers to these queries, is fixed ahead of time. We refer the reader to Soare \cite{Soare-1987} and Downey and Hirschfeldt \cite{DH-2010} for general background on computability theory.

Reductions defined in terms of uniform computability are a central focus in the area of computable analysis, with \emph{Weihrauch reducibility} being arguably the most prominent. See Brattka, Gherardi, and Pauly \cite{bgp21} for a general introduction, and Section \ref{sec:limprobs} below for some of the key definitions. Weihrauch reducibility has been widely used to study relationships between problems in analysis (cf., e.g., \cite{Brattka-2017, BHK-2018, BLMP-2019, HM-2023, HRW-2012}), infinitary combinatorics (cf.~\cite{BR-2017, ddhms16, dghpp20, DPSW-2017, Goh-2020, HJ-2016, Patey-2016a}), topology (cf., e.g.,~\cite{Benham-TA, BDDSV-2024, Genovesi-TA}), and other areas. See Brattka~\cite{Brattka-bib} for a complete bibliography of work in this area. Our interest in this article will be to apply the framework of Weihrauch reducibility to the study of certain problems from finite combinatorics, specifically the finite pigeonhole principle. Although computability theory on finite sets and objects is trivial, uniform computability is not. And, as we will see in Sections \ref{sec:m->n_id_k} and \ref{sec:limprobs}, to gain a more complete picture of the situation surrounding finitary problems requires also looking at infinitary versions.

Part of the motivation for us comes from \emph{reverse mathematics}, an area in the foundation of mathematics that seeks to gauge the logical strength of theorems in terms of the minimal axioms needed to prove them. Standard references are Simpson \cite{Simpson-2009}, Hirschfeldt \cite{Hirschfeldt-2014}, and Dzhafarov and Mummert \cite{DM-2022}. There is a deep connection between computability theory and reverse mathematics (see Shore \cite{Shore-2010}). Recent years have seen a flurry of activity also at the interface of reverse mathematics and Weihrauch reducibility (see, e.g., \cite{ddhms16}, \cite{DPSW-2017}, \cite{GM-2008}, \cite{HJ-2016}). The study of pigeonhole principles in reverse mathematics has a long history, with seminal works by Paris and Kirby \cite{PK-1978}, Hirst \cite{Hirst-1987}, and many others. Our paper is inspired by recent work of Belanger, Chong, Wang, Wong, and Yang \cite{BCWWY-2021}, who showed that the existence of certain highly random binary sequences can be axiomatized in terms of weak pigeonhole principles. These give rise to the principles $\wphp{}{m}{n}$, that we introduce below and investigate here under Weihrauch reducibility.

We begin by stating precisely what we will mean, throughout, by \emph{problems}.

\begin{defn}\
\begin{enumerate}
	\item An \emph{instance-solution problem} (or just \emph{problem}) $\mathsf{P}$ consists of a nonempty set of \emph{instances}, and for each instance $x$, a nonempty set of \emph{solutions} to $x$. When discussing multiple problems, we may qualify the above by referring to \emph{$\mathsf{P}$-instances} and \emph{$\mathsf{P}$-solutions}.
	\item The set of instances of $\mathsf{P}$ is denoted $\operatorname{dom}(\mathsf{P})$, and the set of solutions to $x \in \operatorname{dom}(\mathsf{P})$ is denoted by $\mathsf{P}(x)$.
\end{enumerate}	
\end{defn}

\noindent While quite general, we will only be concerned in problems whose instances and solutions are (or can be coded/represented by) elements of a ``standard space'' like $\mathbb{N}^\mathbb{N}$ (i.e., the set of functions from $\mathbb{N} \to \mathbb{N}$). This includes a great many kinds of objects. For starters, using any standard pairing function we may encode ordered pairs of numbers $(x,y)$ by a single number, denoted $\seq{x,y}$. (See, e.g., \cite{Soare-1987}, p.~3.) We can then extend this to ordered sequences of arbitrary finite length in the obvious way. Thus, in particular, we can code functions $f : \mathbb{N}^n \to \mathbb{N}$ for any $n \geq 1$, and indeed, $f : X^n \to Y$ for any $X,Y \subseteq \mathbb{N}$. We can also code finite and (countably) infinite sequences of such functions.

Throughout, ``integers'' and ``numbers'' will always refer to elements of $\mathbb{N} = \{0,1,2,\ldots\}$. We identify a number $n \geq 1$ with $\{0,\ldots,n-1\}$, so that $j < n$ and $j \in n$ may be used interchangeably. Our focus will center on the following family of problems.

\begin{defn}
For each pair of integers $m > n \geq 2$, the problem $\wphp{}{m}{n}$ has instances all functions $f: m \to n$, and each instance $f$ has solutions all $\{i,j\} \in m^2$ such that $i \neq j$ and $f(i) = f(j)$.
\end{defn}

\noindent Let us call a problem $\mathsf{P}$ \emph{finite} if $\operatorname{dom}(\mathsf{P})$ is a finite subset of $\mathbb{N}$, and $\mathsf{P}(x)$ is a finite subset of $\mathbb{N}$ for each $x \in \operatorname{dom}(\mathsf{P})$. Per our discussion above, each of the problems $\wphp{}{m}{n}$ has a natural representation as a finite problem, and we shall always assume this representation moving forward.

We now define a reducibility notion which, when restricted to finite problems, exactly agrees with so-called \emph{strong} Weihrauch reducibility; we discuss this further in Section \ref{sec:limprobs}. The advantage of using a separate definition here is that we can forego, for now, needing to explicitly mention computability theory.

\begin{defn}\label{defn:gen_reductions}
Let $\mathsf{P}$ and $\mathsf{Q}$ be finite problems.
\begin{enumerate}
	\item $\mathsf{P}$ is \emph{reducible} to $\mathsf{Q}$, written $\mathsf{P} \gred \mathsf{Q}$, if there exist functions $\Phi,\Psi$ with domain and co-domain (possibly proper subsets of) $\N$ such that $\Phi(x) \in \operatorname{dom}(\mathsf{Q})$ whenever $x \in \operatorname{dom}(\mathsf{P})$, and $\Psi({y}) \in \mathsf{P}(X)$ whenever $y \in \mathsf{Q}(\Phi(x))$.
	\item If $\mfP \gred \mfQ$ and $\mfQ \gred \mfP$,
we say that $\mfP$ and $\mfQ$ are \emph{equivalent}, and write $\mfP \gequiv \mfQ$.
\end{enumerate}
\end{defn}
\noindent It is clear that $\gred$ forms a preorder (i.e., it is reflexive and transitive).

The key part of the definition is that the ``backward'' function, $\Psi$, is not allowed to make use of the starting $\mathsf{P}$-instance, $x$. We have the following basic results.

\begin{prop} \label{prop:m->n_basic_rshp}
If $m' \geq m > n \geq n' \geq 2$, then $\wphp{}{m'}{n'} \gred \wphp{}{m}{n}$.
\end{prop}
\begin{proof}
Given $f: m' \to n'$, restrict it to $m$ and regard it as a function from $m$ to $n$. Each solution to $f\restriction m: m \to n$ is also a solution to $f: m' \to n'$. More formally, let $\Phi : \N \to \N$ be the function that, given as input (a code for) a function $f : m' \to n'$, outputs (a code for) $f \restriction m$ (and is, for example, undefined on other inputs). Then $\Phi$ takes instances of $\wphp{}{m'}{n'}$ to instances of $\wphp{}{m}{n}$. Let $\Psi : \N \to \N$ be the identity function. Then $\Phi$ and $\Psi$ witness that $\wphp{}{m'}{n'} \leq \wphp{}{m}{n}$, as was to be shown.
\end{proof}

When proving that a reduction holds, we will typically forego the formal definition of $\Phi$ and $\Psi$ as in the second part of the above proof, and instead directly convert between instances and solutions, as in the first part. When proving a \emph{separation}, i.e., that a reduction does not hold, we will usually need to work with $\Phi$ and $\Psi$ explicitly. In such cases, given an instance $f : m \to n$ of $\wphp{}{m}{n}$, or a solution $\{i,j\}$ to some such $f$, we will write $\Phi(f)$ and $\Psi(\{i,j\})$ for simplicity, even though formally $\Phi$ and $\Psi$ are being applied to the numerical codes of these objects.

We now show that the definition of $\gred$ is non-trivial. Namely, for fixed $n \geq 2$, we show that the hierarchy
\[ \dots \leq \wphp{}{m+1}{n} \leq \wphp{}{m}{n} \leq \dots \leq \wphp{}{n+1}{n} \]
is infinite:

\begin{prop} \label{prop:n^l+1->n_strictly_below_n^l->n}
For all $n,\ell \geq 2$, $\wphp{}{n^\ell+1}{n} < \wphp{}{n^\ell}{n}$.
\end{prop}
\begin{proof}
There are two key facts underlying the separation here.
First, there are $\ell$ many functions $f_1,\dots,f_\ell: n^\ell \to n$ such that no pair $i \neq j$ is a solution to all of $f_1,\dots,f_\ell$. (For example, let $f_k(x)$ be the $k$th digit in the $n$-ary expansion of $x$.)

Second, given $\ell$ many functions $g_1,\dots,g_\ell: n^\ell+1 \to n$, there is some pair $i \neq j$ which is a solution to all of $g_1,\dots,g_\ell$. This follows by applying the pigeonhole principle to the function $i \mapsto (g_1(i),\dots,g_\ell(i))$.

Consider now functions $\Phi$ and $\Psi$ which purport to witness that $\wphp{}{n^\ell}{n} \leq \wphp{}{n^\ell+1}{n}$ as in Definition \ref{defn:gen_reductions}. Fix a common solution $i \neq j$ to the functions $\Phi(f_1),\dots,\Phi(f_\ell)$, where $f_1,\ldots,f_\ell$ are as above. There is then some $k$ such that the pair $\Psi(\{i,j\})$ is not a solution to $f_k$. But then $\{i,j\}$ is a solution to $\Phi(f_k)$ even though $\Psi(\{i,j\})$ is not a solution to $f_k$, a contradiction.
\end{proof}

We will prove (Corollary \ref{cor:m->n_not_below_m'->n'_for_n'<n}) that for fixed $m \geq 3$, the hierarchy $\wphp{}{m}{2} \leq \wphp{}{m}{3} \leq \dots \leq \wphp{}{m}{m-1}$ is strict. In general one would like to know all values of $m$, $n$, $m'$ and $n'$ such that $\wphp{}{m}{n} \leq \wphp{}{m'}{n'}$. We have the following elementary positive result.

\begin{prop} \label{prop:m->n_le_mq+r->nq+r}
For each $q \geq 1$ and all $r,n < m$, $\wphp{}{m}{n} \leq \wphp{}{mq+r}{nq+r}$.
\end{prop}
\begin{proof}
Given $f: m \to n$,
consider $g: mq+r \to nq+r$ defined by
\[
g(x) = 
\begin{cases}
	\lfloor x/m \rfloor n + f(x \text{ mod }m) & \text{if } x < mq,\\
	x-mq+nq & \text{otherwise.}
\end{cases}
\]
If $g(i) = g(j)$ for some $i \neq j$ then we must have
$i, j < mq$, $\lfloor i/m \rfloor = \lfloor j/m \rfloor$, and $f(i \text{ mod }m) = f(j \text{ mod }m)$.
\end{proof}

Many other positive results can be derived from results in the next section;
see Figure \ref{fig:m->n}.
We do not have a conjecture for a complete answer. As we discuss in the next section, any such conjecture would necessarily encompass longstanding open problems in combinatorics.

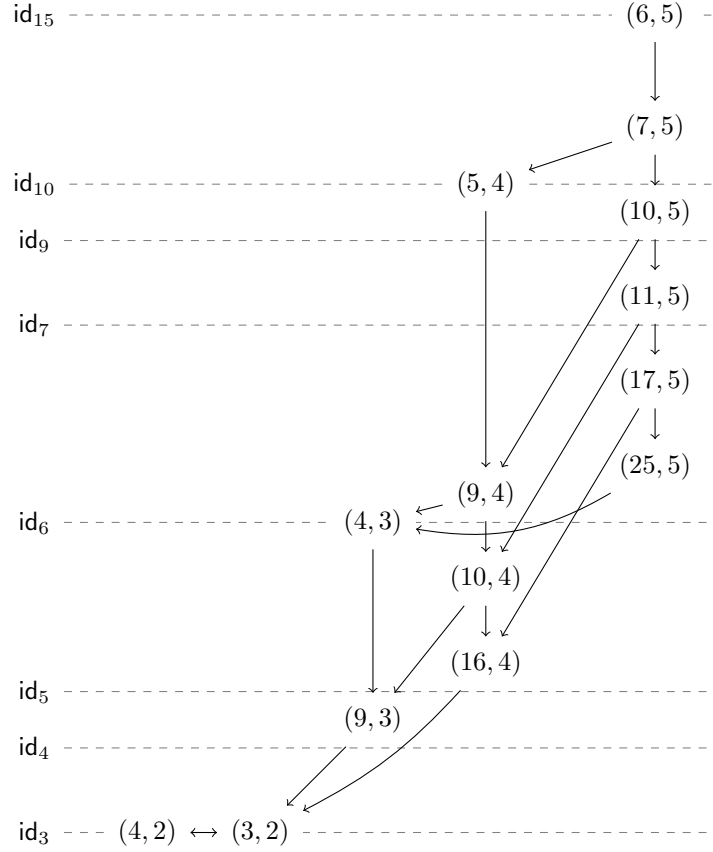
\begin{figure}[ptb]
\centering
\begin{tikzpicture}[scale=1.5,every node/.style={inner sep=5pt,fill=white}]
	\node (id3) at (-2,-0.5) {$\id_3$};
	\draw [dashed,color=gray] (id3) -- (4,-0.5);
	\node (id4) at (-2,0.25) {$\id_4$};
	\draw [dashed,color=gray] (id4) -- (4,0.25);
	\node (id5) at (-2,0.75) {$\id_5$};
	\draw [dashed,color=gray] (id5) -- (4,0.75);
	\node (id6) at (-2,2.25) {$\id_6$};
	\draw [dashed,color=gray] (id6) -- (4,2.25);
	\node (id7) at (-2,4) {$\id_7$};
	\draw [dashed,color=gray] (id7) -- (4,4);
	\node (id9) at (-2,4.75) {$\id_9$};
	\draw [dashed,color=gray] (id9) -- (4,4.75);
	\node (id10) at (-2,5.25) {$\id_{10}$};
	\draw [dashed,color=gray] (id10) -- (4,5.25);
	\node (id15) at (-2,6.75) {$\id_{15}$};
	\draw [dashed,color=gray] (id15) -- (4,6.75);
	\node (32) at (0,-0.5) {$(3,2)$};
	\node (42) at (-1,-0.5) {$(4,2)$};
	\node (43) at (1,2.25) {$(4,3)$};
	\node (93) at (1,0.5) {$(9,3)$};
	\node (54) at (2,5.25) {$(5,4)$};
	\node (94) at (2,2.5) {$(9,4)$};
	\node (104) at (2,1.75) {$(10,4)$};
	\node (164) at (2,1) {$(16,4)$};
	\node (65) at (3.5,6.75) {$(6,5)$};
	\node (75) at (3.5,5.75) {$(7,5)$};
	\node (105) at (3.5,5) {$(10,5)$};
	\node (115) at (3.5,4.25) {$(11,5)$};	
	\node (175) at (3.5,3.5) {$(17,5)$};	
	\node (255) at (3.5,2.75) {$(25,5)$};
	\draw [<->] (32) -- (42);
	\draw [->] (93) -- (32);
	\draw [->] (43) -- (93);
	\draw [->] (54) -- (94);	
	\draw [->] (94) -- (43);
	\draw [->] (94) -- (104);
	\draw [->] (115) -- (104);
	\draw [->] (104) -- (93);
	\draw [->] (104) -- (164);
	\draw [->] (175) -- (164);
	\draw [->] (164) to [bend left=10] (32);
	\draw [->] (65) -- (75);
	\draw [->] (75) -- (105);
	\draw [->] (105) -- (94);	
	\draw [->] (105) -- (115);	
	\draw [->] (115) -- (175);
	\draw [->] (175) -- (255);
	\draw [->] (255) to [bend left=20] (43);
	\draw [->] (75) -- (54);
\end{tikzpicture}
\caption{For $n \leq 5$ and selected $m \leq n^2$, we draw $(m',n') \rightarrow (m,n)$ if $\wphp{}{m}{n} \leq \wphp{}{m'}{n'}$; these follow from Propositions \ref{prop:m->n_basic_rshp}, 
\ref{prop:m->n_le_mq+r->nq+r}, 
\ref{prop:n+1_eq_n+1_choose_2}, 
\ref{prop:3n_optimal},
\ref{prop:id_below_m->n_graph_decomp}
and \ref{prop:6_leq_9->4,10_leq_7->5}.
A node $(m,n)$ lies above (or along) the dotted line passing through $\id_k$ if and only if $\id_k \leq \wphp{}{m}{n}$;
the nonreductions here follow from Lemma \ref{lem:id_k_not_below_m->n_counting_bound}.
}
\label{fig:m->n}
\end{figure}

\section{Relationships between $\wphp{}{m}{n}$ and $\id_k$} \label{sec:m->n_id_k}

In this section we demonstrate a surprising and deep connection between the problems $\wphp{}{m}{n}$ and certain well-known problems in finite combinatorics. This is expressed using the following identity problem on $k$ elements.

\begin{defn}
For each $k \geq 1$, the problem $\id_k$ has instances $j \in \{1,\dots,k\}$, where each instance $j$ has solution set $\{j\}$ (i.e., it has the unique solution $j$).
\end{defn}

\noindent Although $\id_k$ may appear to be a completely trivial problem, it is not so in our context. To see this, suppose for instance that $\id_k \gred \mathsf{P}$, say via $\Phi$ and $\Psi$ as in Definition \ref{defn:gen_reductions}. Then every solution to the $\mathsf{P}$-instance $\Phi(j)$ must directly encode the number $j$. This is because, as noted earlier, $\Psi$ does not have access to what $j$ is, and so can only rely on queries to the $\mathsf{P}$-solution. In this way we gain insight about the coding power of the problem $\mathsf{P}$, particularly if we can show that $\id_k \gred \mathsf{P}$ for some $k$ but not all. As we show below, characterizing what $k$, $m$, and $n$ satisfy $\id_k \gred \wphp{}{m}{n}$ actually turns out to be a difficult problem.

By the pigeonhole principle, $\id_k < \id_\ell$ whenever $k < \ell$. Our starting motivation for considering $\id_k$ stems from the following equivalence.

\begin{prop} \label{prop:n+1_eq_n+1_choose_2}
For each $n \geq 2$, we have $\wphp{}{n+1}{n} \equiv \id_{\binom{n+1}{2}}$.
\end{prop}
\begin{proof}
To prove $\wphp{}{n+1}{n} \leq \id_{\binom{n+1}{2}}$, given $f: n+1 \to n$, we may output the lexicographically least $i < j$ such that $f(i) = f(j)$.

Conversely, view each number in $\binom{n+1}{2}$ as a pair $i < j$. Map it to the function $f: n+1 \to n$ such that $f(i) = f(j) = i$ and $f(k) = k$ otherwise. We can recover $i < j$ from the unique solution $\{i,j\}$ to $f$.
\end{proof}

This suggests we investigate the relationship between $\wphp{}{m}{n}$ and $\id_k$ for general $m > n \geq 2$ and $k \geq 1$. The question of when $\wphp{}{m}{n} \leq \id_k$ is easily settled:

\begin{prop} \label{prop:m->n_not_below_id_n+1_choose_2-1}
$\wphp{}{m}{n} \not\leq \id_{\binom{n+1}{2}-1}$. So $\wphp{}{m}{n} \leq \id_k$ if and only if $k \geq \binom{n+1}{2}$.
\end{prop}
\begin{proof}
Suppose towards a contradiction that $\Phi$ and $\Psi$ witness $\wphp{}{m}{n} \leq \id_{\binom{n+1}{2}-1}$. That means that for all $f: m \to n$, $\Psi(\Phi(f))$ is a solution to $f$. The range of $\Psi$ (restricted to all possible solutions to instances of $\id_{\binom{n+1}{2}-1}$) contains at most $\binom{n+1}{2}-1$ many pairs $\{i,j\}$, where $i < j < m$. View these pairs as the edges of a graph with $m$ vertices. Since there are at most $\binom{n+1}{2}-1$ many edges, this graph is $n$-colorable, say by a coloring $f: m \to n$. Then $\Psi(\Phi(f))$ is not a solution to $f$, contradiction.

The second statement in the proposition then follows from Propositions \ref{prop:m->n_basic_rshp} and \ref{prop:n+1_eq_n+1_choose_2}.
\end{proof}

In particular, we now know that $\wphp{}{m}{n}$ cannot be equivalent to $\id_k$ unless $k = \binom{n+1}{2}$. We will see (Corollary \ref{cor:no_equiv_id_except_4->2}) that the only problems $\wphp{}{m}{n}$ which are equivalent to some $\id_k$ are the problems $\wphp{}{n+1}{n}$ and $\wphp{}{4}{2}$.

\begin{cor} \label{cor:m->n_not_below_m'->n'_for_n'<n}
Whenever $m,m' > n > n' \geq 2$, we have $\wphp{}{m}{n} \not\leq \wphp{}{m'}{n'}$. Therefore $\wphp{}{m}{2} < \wphp{}{m}{3} < \dots < \wphp{}{m}{m-1}$.
\end{cor}
\begin{proof}
By Propositions \ref{prop:m->n_basic_rshp} and \ref{prop:n+1_eq_n+1_choose_2}, $\wphp{}{m'}{n'} \leq \id_{\binom{n'+1}{2}}$. Since $n' < n$, we have $\binom{n'+1}{2} \leq \binom{n+1}{2}-1$ so $\id_{\binom{n'+1}{2}} \leq \id_{\binom{n+1}{2}-1}$. Therefore $\wphp{}{m'}{n'} \leq \id_{\binom{n+1}{2}-1}$, which implies $\wphp{}{m}{n} \not\leq \wphp{}{m'}{n'}$ (Proposition \ref{prop:m->n_not_below_id_n+1_choose_2-1}).\end{proof}

On the other hand, the question of when $\id_k \leq \wphp{}{m}{n}$ is much more difficult (Theorems \ref{thm:latin_sq}, \ref{thm:affine_plane}, \ref{thm:RBIBD}). We begin with a combinatorial characterization:

\begin{lem} \label{lem:id_k_le_m_iff_k_functions_no_common_soln}
The following are equivalent:
\begin{enumerate}
	\item there are $k$ many functions $f_1,\dots,f_k: m \to n$ such that if $\ell \neq \ell'$, then $f_\ell$ and $f_{\ell'}$ have no common solution;
	\item $\id_k \leq \wphp{}{m}{n}$.
\end{enumerate}
\end{lem}
\begin{proof}
If $f_1,\dots,f_k: m \to n$ are as in (1), we can define $\Phi(\ell) = f_\ell$ for each $\ell \in \{1,\dots,k\}$. Given an $\wphp{}{m}{n}$-solution $\{i,j\}$ to some $\Phi(\ell)$, we define $\Psi(\{i,j\})$ to be the unique $\ell$ such that $\{i,j\}$ is a solution to $f_\ell$. (Note that $\Psi$ need not be total.) Conversely, if $\Phi$ and $\Psi$ witness that $\id_k \leq \wphp{}{m}{n}$, then $\Phi(1),\dots,\Phi(k)$ are functions $m \to n$. Furthermore, if $\Phi(\ell)$ and $\Phi(\ell')$ have a common $\wphp{}{m}{n}$-solution $\{i,j\}$, then $\Psi(\{i,j\})$ must be a common $\id_k$-solution to $\ell$ and $\ell'$. This implies $\ell = \ell'$.
\end{proof}

There are two basic reasons why Lemma \ref{lem:id_k_le_m_iff_k_functions_no_common_soln}(2) (and hence (1)) can fail. The first is when $m$ is large relative to $n$, specifically when $m > n^2$. This forces any $f: m \to n$ to have some large fiber, specifically of size $> n$.

\begin{prop} \label{prop:id_2_not_below_n^2+1}
For each $k \geq 2$, we have $\id_k \not\leq \wphp{}{n^2+1}{n}$. More generally, if $n,k \geq 2$ and $m > n^2$, then $\id_k \not\leq \wphp{}{m}{n}$.
\end{prop}
\begin{proof}
If $m > n^2$, then any two functions from $m$ to $n$ have some common solution, so the desired result follows from Lemma \ref{lem:id_k_le_m_iff_k_functions_no_common_soln}.
\end{proof}

The second is when $k$ is large, specifically when $k$ exceeds $\binom{m}{2}$ divided by the minimum number of solutions each $f: m \to n$ can have (calculated below).

\begin{lem} \label{lem:Turan_bound}
If $m = qn + r$ where $0 \leq r < n$,
then each function $f: m \to n$ has at least $q(m-n+r)/2$ many solutions.
Furthermore, if equality holds then $\lfloor m/n \rfloor \leq |f^{-1}(j)| \leq \lceil m/n \rceil$ for every $j < n$.
\end{lem}
\begin{proof}
This result can be proved using convexity of binomial coefficients.
Instead we highlight a connection to extremal graph theory.
Given $f: m \to n$, view its solution set as a graph with $m$ vertices.
The complement of this graph does not contain any clique with $n+1$ vertices.
By Tur\'an's theorem \cite{tur54}, the complement has at most $\frac{(n-1)(m^2-r^2)}{2n} + \binom{r}{2}$ edges. Subtracting this number from $\binom{m}{2}$ yields the desired bound. The second half of the lemma follows from the unique extremal property of the Tur\'an graph.
\end{proof}

The above lemma implies a sufficient numerical condition for $\id_k \not\leq \wphp{}{m}{n}$:

\begin{lem} \label{lem:id_k_not_below_m->n_counting_bound}
If $m = qn + r$ where $0 \leq r < n$, then $\id_k \not\leq \wphp{}{m}{n}$ as long as
\[ k > \frac{m(m-1)}{q(m-n+r)}. \]
Therefore, if $\id_k \leq \wphp{}{m}{n}$, then $k \leq \lfloor \frac{m(m-1)}{q(m-n+r)} \rfloor$. In particular, for each $q \geq 2$, if $\id_k \leq \wphp{}{qn}{n}$, then $k \leq \lfloor \frac{qn-1}{q-1} \rfloor$.
\end{lem}
\begin{proof}
If $k$ satisfies the given bound, then $k \cdot q(m-n+r)/2 > \binom{m}{2}$. By the previous lemma, this implies that among any $k$ functions from $m$ to $n$, at least two of them have some common solution. Conclude by Lemma \ref{lem:id_k_le_m_iff_k_functions_no_common_soln}.
\end{proof}

\begin{cor} \label{cor:n+2_strictly_below_n+1_for_n>2}
For all $n > 2$, we have $\wphp{}{n+2}{n} < \wphp{}{n+1}{n}$.
\end{cor}
\begin{proof}
We have $\frac{(n+2)(n+1)}{1(n+2-n+2)} = \frac{(n+2)(n+1)}{4} < \binom{n+1}{2}$ so $\id_{\binom{n+1}{2}} \not\leq \wphp{}{n+2}{n}$. Conclude by Proposition \ref{prop:n+1_eq_n+1_choose_2}.
\end{proof}

This corollary is false for $n = 2$; see the next corollary. We may now determine exactly when $\wphp{}{m}{n}$ is equivalent to some $\id_k$.

\begin{cor} \label{cor:no_equiv_id_except_4->2}
For all $m > n \geq 2$ and $k \geq 1$, we have $\wphp{}{m}{n} \equiv \id_k$ if and only if (1) $k = \binom{n+1}{2}$ and $m = n+1$, or (2) $n = 2$ and $m = 4$.
\end{cor}
\begin{proof}
For all $n \geq 2$, we have $\wphp{}{m}{n} \leq \id_{\binom{n+1}{2}}$ (Propositions \ref{prop:m->n_basic_rshp}, \ref{prop:n+1_eq_n+1_choose_2}) and $\wphp{}{m}{n} \not\leq \id_{\binom{n+1}{2}-1}$ (Proposition \ref{prop:m->n_not_below_id_n+1_choose_2-1}). So if $\wphp{}{m}{n} \equiv \id_k$, then $k = \binom{n+1}{2}$. If $n > 2$, then by Corollary \ref{cor:n+2_strictly_below_n+1_for_n>2}, $m = n+1$. As for $n = 2$, we have $\id_2 \not\leq \wphp{}{5}{2}$ (Proposition \ref{prop:id_2_not_below_n^2+1}) so $m < 5$, i.e., $m = 3$ or $4$.

Conversely, $\wphp{}{n+1}{n} \equiv \id_{\binom{n+1}{2}}$ (Proposition \ref{prop:n+1_eq_n+1_choose_2}).
Also, it is easy to see that $\id_3 \leq \wphp{}{4}{2}$ (see the beginning of Section \ref{section:adhoc}) so $\wphp{}{3}{2}$ and $\wphp{}{4}{2}$ are both equivalent to $\id_3$.
\end{proof}

\subsection{When $m = n^2$}

Proposition \ref{prop:id_2_not_below_n^2+1} suggests we consider the extreme case $m = n^2$. It turns out that the question of when $\id_k \leq \wphp{}{n^2}{n}$ is equivalent to a longstanding open problem in combinatorics (see \cite[Chapter III.3]{cd07}). Recall that a \emph{Latin square of order $n$} is an $n \times n$ square array, each of whose entries is populated by one of $n$ many symbols, in such a way that each symbol occurs only once in every row and column. 

\begin{thm} \label{thm:latin_sq}
For $k \geq 1$ and $n \geq 2$, the following are equivalent:
\begin{enumerate}
	\item there exist $k$ mutually orthogonal Latin squares of order $n$
	\item $\id_{k+2} \leq \wphp{}{n^2}{n}$.
\end{enumerate}
In particular, $\id_3 \leq \wphp{}{n^2}{n}$ so $\wphp{}{n^2}{n} < \wphp{}{n^2+1}{n}$ by Proposition \ref{prop:id_2_not_below_n^2+1}.
\end{thm}
\begin{proof}
Define functions $f_1, f_2: n^2 \to n$ by $f_1: x \mapsto \lfloor x/n \rfloor$ and $f_2: x \mapsto x \text{ mod } n$. Clearly $f_1$ and $f_2$ have no common solution.

Suppose there are $k$ mutually orthogonal Latin squares of order $n$. We may view each square as a function from $n^2$ to $n$; namely, map $in+j$ to the number in the $i$th row and $j$th column of the Latin square. Denote these functions by $f_3,f_4,\dots,f_{k+2}$. We claim that $f_1,\dots,f_{k+2}$ have no pairwise common solution. Indeed, $f_1$ has no common solution with $f_\ell$ for any $\ell > 2$, because no number appears twice in a single row of a Latin square. Similarly, $f_2$ has no common solution with $f_\ell$ for any $\ell > 2$. As for distinct $\ell,\ell' > 2$, $f_\ell$ and $f_{\ell'}$ have no common solution because the corresponding Latin squares are orthogonal. Therefore $\id_{k+2} \leq \wphp{}{n^2}{n}$.

Conversely, suppose $\id_{k+2} \leq \wphp{}{n^2}{n}$. Fix functions $f_1,\dots,f_{k+2}: n^2 \to n$ with no pairwise common solution. First observe that $f_1 \times f_2$ is an injection from $n^2$ to $n \times n$ (hence a bijection), because $f_1$ and $f_2$ have no common solution. Then, we transform each $f_\ell$ for $\ell > 2$ into a Latin square as follows: In the $i$th row and $j$th column, put the number $f_\ell((f_1 \times f_2)^{-1}(i,j))$. Each row (column  resp.) has no repeated number because $f_1$ ($f_2$ resp.) has no common solution with $f_\ell$. The resulting Latin squares are orthogonal because for distinct $\ell,\ell' > 2$, $f_\ell$ and $f_{\ell'}$ have no common solution.
\end{proof}

Since there exist $n-1$ mutually orthogonal Latin squares of order $n$ if and only if there exists an affine plane of order $n$ (see \cite[Theorem III.3.20]{cd07}), we conclude that:

\begin{thm} \label{thm:affine_plane}
$\id_{n+1} \leq \wphp{}{n^2}{n}$ if and only if there exists an affine plane of order $n$.
\end{thm}

Since affine planes of every prime power order exist (see \cite[Theorem III.3.28]{cd07}):

\begin{cor} \label{cor:id_n+1_below_n^2->n}
If $n$ is a prime power, then $\id_{n+1} \leq \wphp{}{n^2}{n}$.
\end{cor}

\noindent Note that the lefthand side above is optimal: by Lemma \ref{lem:id_k_not_below_m->n_counting_bound} with $q = n$, we have that if $\id_k \leq \wphp{}{n^2}{n}$ then $k \leq \lfloor \frac{n^2-1}{n-1} \rfloor = n+1.$

\subsection{When $2n \leq m < n^2$}

We shall generalize Theorem \ref{thm:affine_plane} by establishing an equivalence between statements of the form $\id_k \leq \wphp{}{qn}{n}$ and the existence of certain \emph{resolvable balanced incomplete block designs}, or \emph{RBIBD} for short (see \cite[Chapter II.7]{cd07}). We recall the definition, for completeness.
\begin{defn}
	Fix $p,m,n \in \mathbb{N}$.
	\begin{itemize}
		\item A $\mathrm{BIBD}(p,n,m)$ consists of a finite set $V$ of \emph{points} of size $p$, and a set of subsets of $V$ called \emph{blocks}, each of size $n$, such that every pair of points lies in exactly $m$ many blocks.
		\item An $\mathrm{RBIBD}(p,n,m)$ is a $\mathrm{BIBD}(p,n,m)$ such that the set of blocks can be partitioned into \emph{parallel classes} in such a way that each point lies in exactly one block in each parallel class.
	\end{itemize}
\end{defn}

\noindent We will take the set of points to always be a subset of $\mathbb{N}$.

We mention that RBIBDs are related to certain universal classes of hash functions (Stinson \cite{stinson94}). Note that the following result generalizes Theorem \ref{thm:affine_plane}, because an affine plane of order $q$ is exactly an $\mathrm{RBIBD}(q^2,q,1)$.

\begin{thm} \label{thm:RBIBD}
Suppose $q-1$ divides $qn-1$. Then $\id_{(qn-1)/(q-1)} \leq \wphp{}{qn}{n}$ if and only if there is an $\mathrm{RBIBD}(qn,q,1)$.
\end{thm}

\begin{proof}
Given a RBIBD$(qn,q,1)$, the number of its blocks is $n\frac{qn-1}{q-1}$, so the number of its parallel classes is $\frac{qn-1}{q-1}$. Each parallel class defines a function from $qn$ to $n$. The resulting set of functions has no pairwise common solution.

Conversely, suppose $\id_{(qn-1)/(q-1)} \leq \wphp{}{qn}{n}$. Fix functions $f_i: qn \to n$ for all $i < (qn-1)/(q-1)$ having no pairwise common solution. By Lemma \ref{lem:Turan_bound}, since $\frac{qn-1}{q-1} \cdot \frac{q(qn-n)}{2} = \binom{qn}{2}$, each $f_i$ has exactly $\frac{q(qn-n)}{2}$ solutions and so each pre-image $f_i^{-1}(j)$ has size $q$. Therefore the blocks $\{f_i^{-1}(j): i < (qn-1)/(q-1), j < n\}$ form an $\mathrm{RBIBD}(qn,q,1)$ with parallel classes $\{f_i^{-1}(j): j < n\}$.
\end{proof}

For $q = 2,3$, we obtain a complete picture by applying classical combinatorial results.

\begin{prop} \label{prop:2n_optimal}
For all $n \geq 2$, $\id_{2n-1} \leq \wphp{}{2n}{n}$ and $\id_{2n} \not\leq \wphp{}{2n}{n}$.
\end{prop}
\begin{proof}
A RBIBD$(2n,2,1)$ is precisely a decomposition of the complete graph $K_{2n}$ into perfect matchings.
Such a decomposition always exists (see \cite[\S VII.5.5]{cd07}).
So $\id_{2n-1} \leq \wphp{}{2n}{n}$ by Theorem \ref{thm:RBIBD}.
We have $\id_{2n} \not\leq \wphp{}{2n}{n}$ by Lemma \ref{lem:id_k_not_below_m->n_counting_bound}.
\end{proof}

\begin{cor} \label{cor:2n_strictly_above_2n+1}
For all $n \geq 2$, $\wphp{}{2n}{n} > \wphp{}{2n+1}{n}$.
\end{cor}
\begin{proof}
$\id_{2n-1} \not\leq \wphp{}{2n+1}{n}$ by Lemma \ref{lem:id_k_not_below_m->n_counting_bound} and the fact that $2n-1 > \frac{(2n+1)2n}{2(n+2)}$. On the other hand, $\id_{2n-1} \leq \wphp{}{2n}{n}$ (Proposition \ref{prop:2n_optimal}).
\end{proof}

\begin{prop} \label{prop:3n_optimal}
If $n \geq 3$ is odd, then $\id_{(3n-1)/2} \leq \wphp{}{3n}{n}$ and $\id_{((3n-1)/2)+1} \not\leq \wphp{}{3n}{n}$.
\end{prop}
\begin{proof}
Since $3n$ is $3$ mod $6$, we may fix a RBIBD$(3n,3,1)$ (see \cite[Theorem II.2.77]{cd07}), better known as a Kirkman triple system of order $3n$.
By Theorem \ref{thm:RBIBD}, $\id_{(3n-1)/2} \leq \wphp{}{3n}{n}$.
By Lemma \ref{lem:id_k_not_below_m->n_counting_bound}, $\id_{((3n-1)/2)+1} \not\leq \wphp{}{3n}{n}$.
\end{proof}

\begin{cor} \label{cor:3n_strictly_above_3n+1_for_n_odd}
If $n \geq 3$ is odd, then $\wphp{}{3n}{n} > \wphp{}{3n+1}{n}$.
\end{cor}
\begin{proof}
$\id_{(3n-1)/2} \not\leq \wphp{}{3n+1}{n}$ by Lemma \ref{lem:id_k_not_below_m->n_counting_bound} and the fact that $\frac{3n-1}{2} > \frac{(3n+1)(3n)}{3(2n+2)}$. On the other hand, $\id_{(3n-1)/2} \leq \wphp{}{3n}{n}$ (Proposition \ref{prop:3n_optimal}).
\end{proof}

We summarize a few of our results in the following theorem.

\begin{thm} \label{thm:summary}
We have $\id_3 \equiv \wphp{}{3}{2} \equiv \wphp{}{4}{2} > \wphp{}{5}{2}$. As for $n > 2$, we have $\id_{\binom{n+1}{2}} \equiv \wphp{}{n+1}{n} > \wphp{}{n+2}{n} \geq \wphp{}{2n}{n} > \wphp{}{2n+1}{n} \geq \wphp{}{n^2}{n} > \wphp{}{n^2+1}{n}$.
\end{thm}
\begin{proof}
Apply Propositions \ref{prop:m->n_basic_rshp}, \ref{prop:n+1_eq_n+1_choose_2}, \ref{prop:id_2_not_below_n^2+1}, Theorem \ref{thm:latin_sq} and Corollaries \ref{cor:2n_strictly_above_2n+1}, \ref{cor:3n_strictly_above_3n+1_for_n_odd}.
\end{proof}

\subsection{When $m \leq 2n+1$}

When $m = 2n$, Proposition \ref{prop:2n_optimal} tells us exactly when $\id_k \leq \wphp{}{m}{n}$. Similar methods (i.e., graph decompositions) yield results for $m \leq 2n+1$:

\begin{prop} \label{prop:id_below_m->n_graph_decomp}
For $2n+1 \geq m > n \geq 2$, let $\alpha(m,n) = (m-1)\left\lfloor \frac{m}{2(m-n)} \right\rfloor$ if $m$ is even and $\alpha(m,n) = \max\left\{\frac{m-1}{2}\left\lfloor \frac{m}{m-n} \right\rfloor, m\left\lfloor \frac{m-1}{2(m-n)} \right\rfloor\right\}$ otherwise. Then $\id_{k(m,n)} \leq \wphp{}{m}{n}$, where $k(m,n) = \max\{\alpha(m',n) \mid m' \in [m,2n+1]\}$. 
\end{prop}
\begin{proof}
Since $\wphp{}{m'}{n} \leq \wphp{}{m}{n}$ whenever $m' \geq m$ (Proposition \ref{prop:m->n_basic_rshp}), it suffices to show that if $m$ and $n$ satisfy $2n \geq m > n \geq 2$, then $\id_{\alpha(m,n)} \leq \wphp{}{m}{n}$. In this proof, let $r$ denote $m-n$.

If $m$ is even, we may decompose $K_m$ into $m-1$ perfect matchings. Each perfect matching has $m/2$ edges, so we may find $\lfloor (m/2)/r \rfloor$ disjoint matchings of size $r$ in it. Each matching of size $r$ yields a function from $m$ to $(m-2r)+r = n$. In total we obtain $(m-1)\lfloor m/2r \rfloor = \alpha(m,n)$ many functions from $m$ to $n$, with no pairwise common solution. By Lemma \ref{lem:id_k_le_m_iff_k_functions_no_common_soln}, $\id_{\alpha(m,n)} \leq \wphp{}{m}{n}$.

On the other hand, if $m$ is odd, we may decompose $K_m$ into $\binom{m}{2}/m = (m-1)/2$ Hamiltonian cycles (see \cite[Remark VII.5.85.1]{cd07}).
Each Hamiltonian cycle has $m$ edges.
If $m = 2n+1$, then we can find a matching of size $r = n+1$ in each Hamiltonian cycle by picking alternate edges.
Similarly, if $m \leq 2n$, since $\lfloor m/r \rfloor \geq 2$, we may find $\lfloor m/r \rfloor$ disjoint matchings of size $r$ in each Hamiltonian cycle, by picking every $\lfloor m/r \rfloor$th edge in the cycle.
In either case, each matching of size $r$ yields a function from $m$ to $n$ and in total we obtain $\frac{m-1}{2}\lfloor m/r \rfloor$ many functions from $m$ to $n$, with no pairwise common solution.

Alternatively, if $m$ is odd, we may decompose $K_m$ into `almost perfect matchings' as follows: First decompose $K_{m+1}$ into $m$ perfect matchings, then remove the added vertex and all incident edges. Each `almost perfect matching' has $(m-1)/2$ edges, so we may find $\lfloor (m-1)/2r \rfloor$ disjoint matchings of size $m-n$ in it. In total we obtain $m \lfloor (m-1)/2r \rfloor$ many functions from $m$ to $n$, with no pairwise common solution.

Combining the previous two paragraphs, we conclude by Lemma \ref{lem:id_k_le_m_iff_k_functions_no_common_soln} that if $m$ is odd, then $\id_{\alpha(m,n)} \leq \wphp{}{m}{n}$.
\end{proof}

The following divisibility conditions on $m$ and $n$ guarantee optimality in Proposition \ref{prop:id_below_m->n_graph_decomp}. (Below, we speak of $n+r$ rather than $m$ in order to simplify the formulas.)

\begin{cor}
Suppose $1 \leq r < n$. If $n$ is $r$ mod $2r$, or if $r$ is odd and $n$ is $0$ or $r+1$ mod $2r$, then $\id_{\binom{n+r}{2}/r} \leq \wphp{}{n+r}{n}$.
\end{cor}

The {conclusion} of the corollary is best possible because if $1 \leq r < n$ and $k > \binom{n+r}{2}/r$, then $\id_k \not\leq \wphp{}{n+r}{n}$ (Lemma \ref{lem:id_k_not_below_m->n_counting_bound}).

\begin{proof}
If $n$ is $r$ mod $2r$, then $n+r$ is even, $2r$ divides $n+r$, and $(n+r-1)\lfloor\frac{n+r}{2r}\rfloor = \binom{n+r}{2}/r$. If $r$ is odd and $n$ is $0$ mod $2r$, then $n+r$ is odd and $\frac{n+r-1}{2}\lfloor \frac{n+r}{r} \rfloor = \binom{n+r}{2}/r$. If $r$ is odd and $n$ is $r+1$ mod $2r$, then $n+r$ is odd, $2r$ divides $n+r-1$, and $(n+r)\lfloor \frac{n+r-1}{2r} \rfloor = \binom{n+r}{2}/r$.
\end{proof}

Proposition \ref{prop:id_below_m->n_graph_decomp} is not optimal in general, as we will see in Proposition \ref{prop:6_leq_9->4,10_leq_7->5}.

\section{Limit problems a.k.a.\ jumps}\label{sec:limprobs}

We turn our attention from the problems $\wphp{}{m}{n}$ to their \emph{jumps} $\wphp{}{m}{n}'$, defined as follows:

\begin{defn} \label{defn:wphp_jump}
For each $m > n \geq 2$, the problem $\wphp{}{m}{n}'$ has instances all (pointwise) convergent sequences of functions $(f_s: m \to n)_{s \in \N}$, with the solutions to any such sequence being all $\{i,j\}$ such that $i \neq j$ and $\lim_s f_s(i) = \lim_s f_s(j)$.
\end{defn}

\noindent The terminology and notation come from computability theory, where one can form, for each set $X$, its \emph{(Turing) jump}, $X'$. This is defined as the set of all (codes for) algorithms $e$, with access to $X$ as an oracle, that halt in finite time on a prescribed input. (Thus, $X'$ is also called \emph{the halting set} relative to $X$.) It is a standard fact of computability that a function $f : \mathbb{N} \to \mathbb{N}$ is computable from $X'$ if and only if there is a function $\hat{f} : \mathbb{N}^2 \to \mathbb{N}$ computable from $X$ such that $\lim_s \hat{f}(x,s)$ exists for all $x$ and is equal to $f(x)$. Here, $\hat{f}$ is called a \emph{limit approximation} to $f$. We can thus think of $\wphp{'}{m}{n}$ as simply the problem $\wphp{}{m}{n}$, but with instances replaced by limit approximations to them, and solutions unchanged.

If we are to consider reductions involving $\wphp{'}{m}{n}$ then Definition \ref{defn:gen_reductions} is no longer suitable. For starters, the functions $\Phi$ and $\Psi$ there have domains and co-domains that are subsets of $\mathbb{N}$, whereas the objects of interest here are no longer just numbers but also functions (and indeed, sequences of functions). As discussed above, we can represent such sequences by elements of $\mathbb{N}^\mathbb{N}$. We thus move to \emph{functionals} $\Phi,\Psi$ with domains and co-domains (possibly proper subsets of) $\mathbb{N}^\mathbb{N}$. We must restrict these to avoid trivialities. Indeed, if we defined $\Phi((f_s)_{s \in \N})$ directly in terms of $\lim_s f_s$,
then no information would be gained by studying $\wphp{'}{m}{n}$ instead of $\wphp{}{m}{n}$.
(Same for $\Psi$.)
A common restriction is for $\Phi$ and $\Psi$ to be continuous with respect to the Baire space topology on $\mathbb{N}^\mathbb{N}$. (This is the topology generated by basic clopen sets of the form $[\alpha] = \{f \in \mathbb{N}^\mathbb{N} : \alpha \preceq f\}$, where $\preceq$ is the intial segment relation, and $\alpha$ ranges over all finite strings of numbers.) A further restriction is for $\Phi$ and $\Psi$ to be \emph{effectively} continuous functionals, also called \emph{Turing functionals}. Informally, $\Phi$ is a Turing functional if there is an algorithm mapping finite strings to finite strings such that $\Phi(f) = g$ if and only if initial segments of $f$ are mapped by the algorithm to initial segments of $g$ in a monotone and consistent way. A standard fact is that $f \leq_{\rm T} g$ if and only if $f = \Phi(g)$ for some Turing functional $\Phi$. (See \cite{Soare-2016}, Chapter 3, for details.)

Other than the jump of $\wphp{}{m}{n}$,
we will also consider the jump of $\id_k$:

\begin{defn}\label{defn:lim_k}
For each $k \geq 2$, the instances of the problem $\mflim_k$ are convergent sequences of numbers $< k$,
with unique solution being the limit of said sequence.
\end{defn}

We now define a notion of reducibility that is suitable for comparing the jump problems.

\begin{defn}
Let $\mathsf{P}$ and $\mathsf{Q}$ be problems.
\begin{enumerate}
	\item $\mfP$ is \emph{strongly Weihrauch reducible} to $\mfQ$,
written $\mfP \lesW \mfQ$,
if there are Turing functionals $\Phi$ and $\Psi$ such that
\begin{itemize}
	\item for every $\mfP$-instance $p$, $\Phi(p)$ is a $\mfQ$-instance, and
	\item for every $\mfQ$-solution $q$ to $\Phi(p)$, $\Psi(q)$ is a $\mfP$-solution to $p$.
\end{itemize}
\item $\mfP$ is \emph{strongly continuous Weihrauch reducible} to $\mfQ$,
written $\mfP \lecsW \mfQ$,
if the above holds for continuous (not necessarily Turing) functionals $\Phi$ and $\Psi$.
\item $\mfP$ is \emph{Weihrauch reducible} to $\mfQ$,
written $\mfP \leW \mfQ$,
if there are Turing functionals $\Phi$ and $\Psi$ such that
\begin{itemize}
	\item for every $\mfP$-instance $p$, $\Phi(p)$ is a $\mfQ$-instance, and
	\item for every $\mfQ$-solution $q$ to $\Phi(p)$, $\Psi(p,q)$ is a $\mfP$-solution to $p$.
\end{itemize}
\item $\mfP$ is \emph{continuous Weihrauch reducible} to $\mfQ$,
written $\mfP \lecW \mfQ$,
if the above holds for continuous (not necessarily Turing) functionals $\Phi$ and $\Psi$.
\end{enumerate}
\end{defn}

\noindent As usual, $\leW$, $\lecW$, $\lesW$ and $\lecsW$ each form a preorder.
We write $\mfP \equiv \mfQ$ with the appropriate subscript and/or superscript
if $\mfP$ and $\mfQ$ are reducible to each other.

It is easy to see that for finite problems in the sense used in Section \ref{sec:intro}, the reducibility $\gred$ from Definition \ref{defn:gen_reductions} coincides with both $\lesW$ and $\lecsW$. Not so, however, for $\leW$ or $\lecW$. Indeed, under both $\leW$ and $\lecW$ the problems $\wphp{}{m}{n}$ all collapse because access to $f: m \to n$ is sufficient for uniformly computing a solution for $f$ (simply by checking all possible pairs of numbers below $m$).
We will see (e.g., from the following theorem and previous separations) that this is not the case for the jumps $\wphp{'}{m}{n}$. The following theorem lifts all previous results about $\wphp{}{m}{n}$ to analogous results about $\wphp{'}{m}{n}$. We defer its proof to the next section.

\begin{thm} \label{thm:reduction_lifting}
The following are equivalent:
\begin{enumerate}
	\item $\mflim_k \leW \wphp{'}{m}{n}$
	\item $\mflim_k \lesW \wphp{'}{m}{n}$
	\item $\mflim_k \lecsW \wphp{'}{m}{n}$
	\item $\id_k \lecsW \wphp{}{m}{n}$
	\item $\id_k \gred \wphp{}{m}{n}$
	\item $\id_k \lesW \wphp{}{m}{n}$.
\end{enumerate}
Similarly, $\wphp{}{m_0}{n_0} \leq \wphp{}{m_1}{n_1}$ if and only if $\wphp{'}{m_0}{n_0} \leW \wphp{'}{m_1}{n_1}$.
\end{thm}

The above theorem allows us to separate $\wphp{}{m_0}{n_0}$ and $\wphp{}{m_1}{n_1}$ by separating their jumps. The latter appears more tractable because the jump operation ``amplifies the distance'' between problems. For example, since $\id_2 \not\leq \wphp{}{n^2+1}{n}$ (Proposition \ref{prop:id_2_not_below_n^2+1}), the $\id_k$-hierarchy does not separate $\wphp{}{n^2+1}{n}$, $\wphp{}{n^2+2}{n}$, etc. Instead, we will establish some separations between the latter problems by comparing them with other hierarchies of problems, such as the \emph{all-or-unique-choice} problems:

In the sequel, we use the following notation. Fix $k,j,n,m \in \N$. We let $k^n$ denote the finite sequence $p : n \to k$ with $p(s) = k$ for all $s < n$, and $k^\N$ the infinite sequence $p \in \N^\N$ with $p(s) = k$ for all $s$. We let $k^n j^m$ and $k^n j^\N$ denote the concatenation of $k^n$ by $j^m$ and $j^\N$, respectively. We identify $j$ with $j^1$ in this context, so that, e.g., $k^n j$ is the same as $k^n j^1$, etc.

\begin{defn}
For each $k \geq 2$, the problem $\AoUC_k$ has two types of instances $p \in \N^\N$: Either there is a \emph{unique} $\ell < k$ that does \emph{not} appear in $p$,
in which case $p$ has unique solution $\ell$,
or $p = k^\N$,
in which case any $j < k$ is a solution to $p$.
\end{defn}

The problems $\AoUC_k$ are strictly ascending in $k$ under Weihrauch reducibility \cite[Proposition 7.10]{brattka21}.
It follows that $\AoUC_k \sleW \mflim_2$ for all $k$:
To reduce $\AoUC_k$ to $\mflim_2$, use $\mflim_2$ to decide whether any $\ell < k$ ever appears in $p$.

We begin with a combinatorial characterization analogous to Lemma \ref{lem:id_k_le_m_iff_k_functions_no_common_soln}.

\begin{thm} \label{thm:AoUC_k_leW_m_iff_k+1_functions_no_triangle}
Suppose $k \geq 2$ and $m > n \geq 2$.
The following are equivalent:
\begin{enumerate}
	\item there exist $g,f_0,\dots,f_{k-1}: m \to n$ such that for all distinct $\ell,\ell' < k$, the functions $f_\ell$, $f_{\ell'}$, and $g$ have no $\wphp{}{m}{n}$-solution in common;
	\item $\AoUC_k \leW \wphp{'}{m}{n}$;
	\item $\AoUC_k \lecW \wphp{'}{m}{n}$.
\end{enumerate}
\end{thm}

\begin{proof}
To prove that (1) implies (2), first suppose we have functions $g,f_0,\dots,f_{k-1}: m \to n$ as above.
Given an $\AoUC_k$-instance $p \in \N^\N$,
we compute a sequence $(h_s: m \to n)_{s \in \N}$ as follows: for each $s$, if there is a unique $\ell < k$ which does not appear in $p\restriction s$,
we let $h_s = f_\ell$; otherwise, $h_s = g$.
Thus, $\lim_s h_s = f_\ell$ for some $\ell < k$ if and only if $\ell$ is the unique number below $k$ that does not appear in $p$.

Now let $\{i,j\}$ be any $\wphp{}{m}{n}$-solution to $\lim_s h_s: m \to n$. We compute an $\AoUC_k$-solution for $p$ as follows.
First, if $\{i,j\}$ is \emph{not} an $\wphp{}{m}{n}$-solution to $g$, then necessarily $\lim_s h_s \neq g$, so there is a unique $\ell < k$ that does not appear in $p$. We can compute $\ell$ from $p$ by searching for an $s \in \N$ such that $\ell$ is the unique number smaller than $k$ which does not appear in $p\restriction s$. We then output $\Psi(p,\{i,j\}) = \ell$.

If, on the other hand, $\{i,j\}$ is a solution to $g$,
then by assumption there is at most one $\ell < k$ such that $\{i,j\}$ is also a solution to $f_\ell$. Note that we can computably determine, from $\{i,j\}$, whether there is such an $\ell$. If so, $\lim_s h_s$ is either $g$ or $f_\ell$, so in either case $\ell$ is a solution to $p$, and we can define $\Psi(p,\{i,j\}) = \ell$. If no such $\ell$ exists then necessarily $\lim_s h_s = g$, in which case we can simply set $\Psi(p,\{i,j\}) = 0$. Thus, $\AoUC_k \leW \wphp{'}{m}{n}$.

The implication from (2) to (3) is trivial.

To show that (3) implies (1), suppose $\AoUC_k \lecW \wphp{'}{m}{n}$ as witnessed by continuous functionals $\Phi$ and $\Psi$.
Then $\Phi(k^\N)$ is a limit approximation to a function $h_k: m \to n$.
For each solution $\{i,j\}$ to $h_k$,
$\Psi(k^\N,\{i,j\})$ outputs some number smaller than $k$.
Since $\Psi$ is a continuous map to a discrete space and there are only finitely many such pairs $\{i,j\}$,
we may fix $s \in \N$ large enough so that for all such $i \neq j$ and all $p \in \N^\N$ extending $k^s$,
$\Psi(p,\{i,j\})$ is defined and equals $\Psi(k^\N,\{i,j\})$.

For each $\ell < k$, let $p_\ell$ be the sequence obtained by deleting $\ell$ from
\[ k^s 1 2 \cdots (k-1)k^\N. \]
Each $p_\ell$ is an $\AoUC_k$-instance with unique solution $\ell$,
so $\Phi(p_\ell)$ is a limit approximation to a function $f_\ell: m \to n$.
Furthermore, $\Psi(p_\ell,\{i,j\})$ is defined and equals $\ell$ whenever $\{i,j\}$ is an $\wphp{}{m}{n}$-solution to $f_\ell$. To complete the proof, we show that for all $\ell \neq \ell' < k$, the functions
$f_\ell$, $f_{\ell'}$, and $g$ have no solution in common.
Indeed, suppose $\{i,j\}$ is a solution to $f_\ell$ and $g$. By choice of $s$, it follows that $\Psi(k^s,\{i,j\}) = \Psi(k^\N,\{i,j\})$, and since $p_{\ell'}$ extends $k^s$ for all $\ell' < k$, it also follows that $\Psi(p_{\ell'},\{i,j\}) = \Psi(k^s,\{i,j\}) = \Psi(p_\ell,\{i,j\}) = \ell$. But now by correctness of $\Psi$, we see that $\{i,j\}$ cannot be a $\wphp{}{m}{n}$-solution to $f_{\ell'}$ for any $\ell' < k$ different from $\ell$, because $\ell$ is not an $\AoUC_k$-solution to $p_{\ell'}$.
\end{proof}

Recall from Proposition \ref{prop:id_2_not_below_n^2+1} (and Theorem \ref{thm:reduction_lifting}) that $\mflim_2 \not\lecW \wphp{'}{n^2+1}{n}$.
We shall improve this result using Theorem \ref{thm:AoUC_k_leW_m_iff_k+1_functions_no_triangle}.

\begin{prop} \label{prop:AoUC_not_below_n^2+1->n}
$\AoUC_{\binom{n+1}{2}+1} \not\lecW \wphp{'}{n^2+1}{n}$ for all $n \geq 2$.
\end{prop}
\begin{proof}
Consider functions $g,f_0,\dots,f_{\binom{n+1}{2}}: n^2+1 \to n$.
By the pigeonhole principle, fix $S \subseteq n^2+1$ such that $g$ is constant on $S$ and $|S| = n+1$.
For each $\ell \leq \binom{n+1}{2}$, $f_\ell$ must have a solution $\{i,j\}$ where $i,j \in S$.
Since there are only $\binom{n+1}{2}$ many pairs of distinct numbers from $S$,
there must be $\ell \neq \ell'$ and $\{i,j\}$ (from $S$)
such that $\{i,j\}$ is a solution for both $f_\ell$ and $f_{\ell'}$.
\end{proof}

The previous proposition is sharp because $\AoUC_{\binom{n+1}{2}+1} \leW \mflim_2 \leW \wphp{'}{n^2}{n}$ (Theorems \ref{thm:latin_sq}, \ref{thm:reduction_lifting})
and $\AoUC_{\binom{n+1}{2}} \leW \wphp{'}{n^2+1}{n}$ (Corollary \ref{cor:AoUC_n+1_choose_2_below_n^2+n->n} later).

\begin{prop} \label{prop:id_k_leW_m->n_implies_AoUC_k_leW_mn->n}
For $k \geq 2$ and $n^2 \geq m > n \geq 2$,
if $\id_k \leq \wphp{}{m}{n}$,
then $\AoUC_k \leW \wphp{'}{mn}{n}$.
\end{prop}
\begin{proof}
We shall use Lemma \ref{lem:id_k_le_m_iff_k_functions_no_common_soln} and Theorem \ref{thm:AoUC_k_leW_m_iff_k+1_functions_no_triangle}.
Fix functions $h_0,\dots,h_{k-1}: m \to n$ with no common solution.
Define $g: mn \to n$ by $g(x) = \lfloor x/m \rfloor$.
For each $\ell < k$,
define $f_\ell: mn \to n$ by $f_\ell(x) = h_\ell(x \text{ mod }m)$. We claim there are no $\ell \neq \ell'$ such that $g$, $f_\ell$, and $f_{\ell'}$ have a common solution.
Indeed, suppose $g(i) = g(j)$, $f_\ell(i) = f_\ell(j)$, and $f_{\ell'}(i) = f_{\ell'}(j)$.
Then $h_\ell(i\text{ mod }m) = h_\ell(j\text{ mod }m)$ and $h_{\ell'}(i\text{ mod }m) = h_{\ell'}(j\text{ mod }m)$.
Since $h_\ell$ and $h_{\ell'}$ have no common solution, $i$ and $j$ must be equal mod $m$.
Since $g(i) = g(j)$, we also have $\lfloor i/m \rfloor = \lfloor j/m \rfloor$, and so $i = j$ as desired.
\end{proof}	

By Proposition \ref{prop:n+1_eq_n+1_choose_2} and Theorem \ref{thm:latin_sq} respectively, we obtain two corollaries.

\begin{cor} \label{cor:AoUC_n+1_choose_2_below_n^2+n->n}
$\AoUC_{\binom{n+1}{2}} \leW \wphp{'}{n^2+n}{n}$ for all $n \geq 2$.
\end{cor}

\begin{cor} \label{cor:AoUC_3_below_n^3->n}
$\AoUC_3 \leW \wphp{'}{n^3}{n}$ for all $n \geq 2$. (For $n = 2$ this strengthens the previous corollary, but we will strengthen it further in Proposition \ref{prop:C_3_leq_8->2}.)
\end{cor}

To calibrate the strength of $\wphp{}{m}{n}$ for $m > n^3$, we consider the family of \emph{all-or-co-unique-choice} problems:

\begin{defn}
For each $k \geq 2$, the problem $\ACC_k$ has two types of instances $p \in \N^\N$:
Either there is some \emph{unique} $\ell < k$ which \emph{does} appear in $p$,
in which case any $j < k$ other than $\ell$ is a solution to $p$,
or $p = k^\N$,
in which case any $j < k$ is a solution to $p$.
\end{defn}

$\ACC_2$ and $\AoUC_2$ are clearly equal.
Furthermore, the problems $\ACC_k$ are strictly \emph{descending} under Weihrauch reducibility (see \cite[Fact 3.4]{bhkr17}).
Therefore
\[ \dots \sleW \ACC_3 \sleW \ACC_2 = \AoUC_2 \sleW \AoUC_3 \sleW \dots \sleW \mflim_2. \]

Again we have a combinatorial characterization akin to Lemma \ref{lem:id_k_le_m_iff_k_functions_no_common_soln} and Theorem \ref{thm:AoUC_k_leW_m_iff_k+1_functions_no_triangle}. We omit the proof, which is quite similar to that of Theorem \ref{thm:AoUC_k_leW_m_iff_k+1_functions_no_triangle}.

\begin{thm} \label{thm:ACC_k_leW_m_iff_k+1_functions_no_common_soln}
Suppose $k \geq 2$ and $m > n \geq 2$.
The following are equivalent:
\begin{enumerate}
	\item there exist  $g,f_0,\dots,f_{k-1}: m \to n$
	with no $\wphp{}{m}{n}$-solution in common;
	\item $\ACC_k \leW \wphp{'}{m}{n}$;
	\item $\ACC_k \lecW \wphp{'}{m}{n}$.
\end{enumerate}
\end{thm}

\begin{prop} \label{prop:ACC_below_n^k+1}
$\ACC_k \leW \wphp{'}{n^{k+1}}{n}$ for all $k \geq 2$.
(For $k=2$, Corollary \ref{cor:AoUC_3_below_n^3->n} provides a stronger result.)
\end{prop}
\begin{proof}
Use Theorem \ref{thm:ACC_k_leW_m_iff_k+1_functions_no_common_soln} and the fact
(already mentioned in the proof of Proposition \ref{prop:n^l+1->n_strictly_below_n^l->n})
that there are $k+1$ many functions $f_0,\dots,f_k: n^{k+1} \to n$ such that no pair $i \neq j$ is a solution to all of $f_0,\dots,f_{k}$.
\end{proof}

On the other hand:

\begin{prop} \label{prop:ACC_not_below_n^(k+1)+1}
$\ACC_k \not\lecW \wphp{'}{n^{k+1}+1}{n}$ for all $k \geq 2$.
\end{prop}
\begin{proof}
Use Theorem \ref{thm:ACC_k_leW_m_iff_k+1_functions_no_common_soln} and the fact
(already mentioned in the proof of Proposition \ref{prop:n^l+1->n_strictly_below_n^l->n})
that given $k+1$ many functions $f_0,\dots,f_k: n^{k+1}+1 \to n$, there is some pair $i \neq j$ which is a solution to all of $f_0,\dots,f_{k}$.
\end{proof}

We now have another proof of Proposition \ref{prop:n^l+1->n_strictly_below_n^l->n}:
Combine Propositions \ref{prop:AoUC_not_below_n^2+1->n}, \ref{prop:ACC_below_n^k+1}, \ref{prop:ACC_not_below_n^(k+1)+1} and Theorem \ref{thm:reduction_lifting}.
While this proof is longer, it has the advantage of being modular;
instead of directly separating $\wphp{'}{n^{k+1}}{n}$ and $\wphp{'}{n^{k+1}+1}{n}$,
we have identified a natural problem ($\ACC_k$) which reduces to the former but not the latter.
We summarize most of this section's results in Figure \ref{fig:m->n_jumps}.

\begin{figure}[ptb]
\centering
\begin{tikzpicture}[scale=1,every node/.style={inner sep=3pt,fill=white}]
	\node [anchor=east] (acc4) at (0,0) {$\ACC_4$};
	\draw [dashed,color=gray] (acc4) -- (2,0);
	\node [anchor=east] (acc3) at (0,1.1) {$\ACC_3$};
	\draw [dashed,color=gray] (acc3) -- (2,1.1);
	\node [anchor=east] (acc2) at (0,2.3) {$\AoUC_2 = \ACC_2$};
	\draw [dashed,color=gray] (acc2) -- (2,2.3);
	\node [anchor=east] (auc_n+1_choose_2) at (0,2.8) {$\AoUC_{\binom{n+1}{2}}$};
	\draw [dashed,color=gray] (auc_n+1_choose_2) -- (2,2.8);
	\node [anchor=east] (auc_n+1_choose_2+1) at (0,3.9) {$\AoUC_{\binom{n+1}{2}+1}$};
	\draw [dashed,color=gray] (auc_n+1_choose_2+1) -- (2,3.9);
	\node [anchor=east] (lim2) at (0,4.5) {$\mflim_2$};
	\node [anchor=east] (lim3) at (0,5) {$\mflim_3$};
	\draw [dashed,color=gray] (lim3) -- (2,5);
	\node [anchor=west] (n5) at (0.5,0.25) {$n^5$};
	\node [anchor=west] (n4+1) at (0.5,0.75) {$n^4+1$};
	\node [anchor=west] (n4) at (0.5,1.5) {$n^4$};
	\node [anchor=west] (n3+1) at (0.5,2) {$n^3+1$};
	\node [anchor=west] (n2+n) at (0.5,3.1) {$n^2+n$};
	\node [anchor=west] (n2+1) at (0.5,3.6) {$n^2+1$};
	\node [anchor=west] (n2) at (0.5,5.3) {$n^2$};
\end{tikzpicture}
\caption{We describe the strength of $\wphp{'}{m}{n}$ for $m$ between $n^2$ and $n^5$. A node $m$ lies above the dotted line passing through $\ACC_k$/$\AoUC_k$/etc.\
if and only if the latter is Weihrauch reducible to $\wphp{'}{m}{n}$,
according to Theorems \ref{thm:latin_sq},
\ref{thm:reduction_lifting},
Proposition \ref{prop:AoUC_not_below_n^2+1->n},
Corollary \ref{cor:AoUC_n+1_choose_2_below_n^2+n->n},
Propositions \ref{prop:ACC_below_n^k+1} and
\ref{prop:ACC_not_below_n^(k+1)+1}.
}
\label{fig:m->n_jumps}
\end{figure}
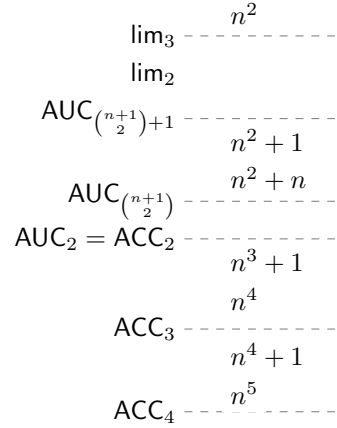

It is natural to wonder if one can carry out analogous reasoning for $\wphp{}{m}{n}$ instead of $\wphp{'}{m}{n}$.
For example, if $\mfP$ is a problem such that $\mfP' \eqW \ACC_k$,
then we could perhaps use $\mfP$ to separate $\wphp{}{n^{k+1}}{n}$ and $\wphp{}{n^{k+1}+1}{n}$,
without considering their jumps.
This is not possible by results of \cite{bgm12,lrp15},
as we explain below.
To this end we consider the family of \emph{choice} problems:

\begin{defn}
For each $k \geq 2$, instances of the problem $\C_k$ are sequences $p \in \N^\N$ so that at least one $\ell < k$ does not appear in $p$, and a solution to $p$ is any such $\ell$.
\end{defn}

The problems $\C_k$ are strictly ascending under Weihrauch reducibility (see \cite[Proposition 7.18]{bgp21}).
It is clear that $\AoUC_k \leW \C_k$ for all $k \geq 2$.
Equality holds for $k = 2$,
while for $k \geq 3$,
$\AoUC_k \sleW \C_k$ because $\AoUC_k \leW \mflim_2$ while $\C_k \not\leW \mflim_2$, as is easily seen (see remarks before Proposition \ref{prop:C_3_leq_8->2}).

Say a problem $\mathsf{P}$ with instances and solutions in $\mathbb{N}^\N$ is \emph{computable} if $\mathsf{P}$ is Weihrauch reducible to the identity problem on $\N^\N$ (i.e., the problem having all $p \in \N^\N$ as instances, and each such $p$ having itself as its unique solution).

\begin{prop}
If $\mfP' \leW \C_k$ for some $k \geq 2$, then $\mfP'$ is computable.
\end{prop}
\begin{proof}
Every jump $\mfP'$ is a fractal \cite[Definition 2.5, Proposition 5.8]{bgm12}
yet every fractal which is Weihrauch reducible to $\C_k$ is computable \cite[Theorem 2.3]{lrp15}.
\end{proof}

Since $\ACC_k$ and $\AoUC_k$ are not computable but are Weihrauch reducible to $\C_k$,
it follows that they are not equivalent to the jump of any problem.
(It also follows that the reductions in
Corollaries \ref{cor:AoUC_n+1_choose_2_below_n^2+n->n},
\ref{cor:AoUC_3_below_n^3->n} and
Proposition \ref{prop:ACC_below_n^k+1}
are strict.)

We give a combinatorial characterization of when $\C_k$ is reducible to $\wphp{'}{m}{n}$:

\begin{thm} \label{thm:C_k_leW_m_iff_functions_indexed_by_injections}
The following properties are equivalent:
\begin{enumerate}
	\item There are functions $f_\sigma: m \to n$ ($\sigma$ ranging over injections from a proper initial segment of $k$ to $k$ itself)
such that if $\{i,j\}$ is a solution of $f_\sigma$,
there exists $\ell < k$ such that if $\tau$ extends $\sigma$
and $\{i,j\}$ is a solution of $f_\tau$,
then $\ell$ is not in the range of $\tau$;
\item $\C_k \leW \wphp{'}{m}{n}$;
	\item $\C_k \lecW \wphp{'}{m}{n}$.
\end{enumerate}
\end{thm}
\begin{proof}
We show that (1) implies (2), and that (3) implies (1). That (2) implies (3) is clear. So suppose we have such functions $f_\sigma: m \to n$ as in (1). We define Turing functionals $\Phi$ and $\Psi$ to witness that $\C_k \leW \wphp{'}{m}{n}$. Fix a $\C_k$-instance $p \in \N^\N$. For $s \in \mathbb{N}$, define $\sigma_{p \restriction s} : |\operatorname{ran}(p \restriction s)| \to k$ to be the map taking each $i < |\operatorname{ran}(p \restriction s)|$ to the $i$th element of the range of $p \restriction s$ in the order it appears there, ignoring all numbers $\geq k$. Thus, $\sigma_{p \restriction s}$ is an injection from a proper initial segment of $k$ to $k$ itself, and if $t \geq s$ then $\sigma_{p \restriction t}$ is equal to $\sigma_{p \restriction s}$ or extends it as a finite function. Let $\sigma_p = \lim_s \sigma_{p \restriction s}$, and note that for all $s$, $\operatorname{ran}(\sigma_{p \restriction s}) \subseteq \operatorname{ran}(\sigma_p) = \operatorname{ran}(p) \cap k$. Define $\Phi(p) = (f_{\sigma_{p \restriction s}})_s$, which is an $\wphp{'}{m}{n}$-instance since $f_{\sigma_{p \restriction s}} = f_{\sigma_p} : m \to n$ for all sufficiently large $s$. 

Now suppose $\{i,j\}$ is any $\wphp{'}{m}{n}$-solution to $(f_{\sigma_{p \restriction s}})_s$, or equivalently, any $\wphp{}{m}{n}$-solution to $f_{\sigma_p}$. We define $\Psi(p,\{i,j\})$ as follows. First, use $p$ to enumerate $\sigma_{p\restriction s}$ until we find some $s$
such that $f_{\sigma_{p\restriction s}}$ has $\{i,j\}$ as a solution. (Such an $s$ exists since, in particular, any $s$ so that $f_{\sigma_{p \restriction s}} = f_{\sigma_p}$ will do.) Then define $\Psi(p,\{i,j\})$ to be the least $\ell < k$ such that
if $\tau$ extends $\sigma_{p\restriction s}$ and $\{i,j\}$ is a solution of $f_\tau$,
then $\ell$ is not in the range of $f_\tau$.
(Such $\ell$ exists by assumption.
Given $\{i,j\}$ and $p\restriction s$, since there are only finitely many extensions of $\sigma_{p \restriction s}$,  we can computably check each $\ell < k$ to see if satisfies the above condition, and thereby eventually find $\ell$ computably.) Note that $\sigma_p$ extends (or is equal to) $\sigma_{p\restriction s}$.
So $\ell$ is not in the range of $\sigma_p$, and hence also not in the range of $p$.
In other words, $\ell$ is a $\C_k$-solution to $p$,
completing the construction of a reduction from $\C_k$ to $\wphp{'}{m}{n}$.

Next, assume (3). Let $\Phi$ and $\Psi$ be continuous functionals witnessing that $\C_k \lecW \wphp{'}{m}{n}$. For each injection $\sigma$ from a proper initial segment of $k$ to $k$ itself,
we shall define (inductively) a $\C_k$-instance $p_\sigma \in \N^\N$ with solution set $k \setminus \operatorname{ran}(\sigma)$ and a function $f_\sigma: m \to n$ as follows.

For the empty injection $\emptyset$,
define $p_\emptyset = k^\N$, which is a $\C_k$-instance with solution set $k$.
Whenever $p_\sigma$ has been defined,
define $f_\sigma: m \to n$ to be the limit of $\Phi(p_\sigma)$.
Having defined $p_\sigma$ and $f_\sigma$,
assuming that more than one number $\ell < k$ does not appear in $p_\sigma$,
consider each such $\ell$.
(If only one such $\ell$ exists,
we cannot extend $\sigma$ to an injection from a proper initial segment of $k$ to $k$ itself,
so this branch of the construction ends.)
Let $\sigma \ell$ denote the injection which extends $\sigma$ by mapping the least number not in the domain of $\sigma$ to $\ell$.
Define a $\C_k$-instance $p_{\sigma \ell}$ as follows.
Since $\Psi$ is a continuous map to a discrete space and $f_\sigma$ has only finitely many solutions,
fix the shortest initial segment $p_\sigma\restriction s$
(with range equal that of $p_\sigma$)
such that for every solution $\{i,j\}$ of $f_\sigma$ and every $q \in \N^\N$ extending $p_\sigma\restriction s$,
$\Psi(q,\{i,j\})$ is defined and equals $\Psi(p_\sigma,\{i,j\})$.
Define $p_{\sigma \ell} = (p_\sigma \restriction s)\ell k^\N$.
By assumption on $\sigma$, $p_\sigma$, and $\ell$,
$p_{\sigma \ell}$ is a $\C_k$-instance with solution set $k \setminus \operatorname{ran}(\sigma\ell)$.
As mentioned above we define $f_{\sigma\ell} = \lim \Phi(p_{\sigma \ell})$.
This completes the construction of all $p_\sigma$ and $f_\sigma$.

Suppose $\{i,j\}$ is a solution of $f_\sigma$.
Since $f_\sigma = \lim \Phi(p_\sigma)$,
$\Psi(p_\sigma,\{i,j\})$ must be a $\C_k$-solution for $p_\sigma$.
Denote it by $\ell$.
We claim that if $\tau$ extends $\sigma$ and $\{i,j\}$ is a solution of $f_\tau$,
then $\ell$ is not in the range of $\tau$.
Since $\{i,j\}$ is a solution of $f_\tau = \lim \Phi(p_\tau)$,
$\Psi(p_\tau,\{i,j\})$ is a $\C_k$-solution of $p_\tau$.
Since $\tau$ extends $\sigma$ and $\{i,j\}$ is a solution of $f_\sigma$,
$\Psi(p_\tau,\{i,j\})$ is defined and equals $\ell$ (by construction of $p_\tau$).
We conclude that $\ell$ is a $\C_k$-solution of $p_\tau$,
i.e., $\ell$ is not in the range of $\tau$.
\end{proof}

In Section \ref{section:adhoc}, we shall use Theorem \ref{thm:C_k_leW_m_iff_functions_indexed_by_injections} to strengthen Corollary \ref{cor:AoUC_3_below_n^3->n} for $n = 2$ (by replacing $\AoUC_3$ with the stronger $\C_3$).

\section{Lifting reductions}

In this section we prove Theorem \ref{thm:reduction_lifting},
which largely consists of combining results from the literature.
In order to state said results,
we first need to define the jump of an arbitrary problem.

\begin{defn} \label{defn:jump}
Given a problem $\mfP$ whose instances and solutions are elements of $\N^\N$,
its \emph{jump} $\mfP'$ is the problem whose instances are sequences $(p_s)_{s \in \N}$ in $\N^\N$ which converge to some $\mfP$-instance,
with solutions being $\mfP$-solutions to $\lim_s p_s$.
\end{defn}

There is an important distinction between this definition and the specific jump problems defined in the previous section. Namely, in a $\mfP'$-instance $(p_s)_{s \in \N}$, it \emph{need not} be the case that {each} $p_s$ is itself a $\mfP$-instance; this only has to be true of the limit. This presents an apparent conflict of notation. However, for all the problems we have considered and will consider in this paper, the two definitions are equivalent up to strong Weihrauch reducibility. Thus, for all our intents and purposes we can treat them as being the same. In particular, we have:
\begin{itemize}
    \item $\id'_k \eqsW \mflim_k$;
    \item the jump of $\wphp{}{m}{n}$ in the sense of Definition \ref{defn:jump} is strongly Weihrauch equivalent to $\wphp{'}{m}{n}$ in the sense of Definition \ref{defn:wphp_jump}.
\end{itemize}
This also justifies our comment above Definition \ref{defn:lim_k}, that $\mflim_k$ is the jump of $\id_k$.

Our proof of Theorem \ref{thm:reduction_lifting} will use the following results relating $\leW$, $\lesW$, $\lecsW$ and jumps.
First, strong Weihrauch reductions lift to jumps:

\begin{prop}[{Brattka, Gherardi, Marcone \cite[Proposition 5.6]{bgm12}}] \label{prop:jump_sW}
If $\mfP \lesW \mfQ$, then $\mfP' \lesW \mfQ'$.
\end{prop}

The converse is true for \emph{continuous} strong Weihrauch reductions:

\begin{thm}[{Brattka, Gherardi, Pauly \cite[Theorem 6.11]{bgp21}, following Brattka, H\"olzl, Kuyper \cite[Theorem 11]{bhku17}}] \label{thm:jump_inversion}
If $\mfP \not\lecsW \mfQ$, then $\mfP' \not\lecsW \mfQ'$.
\end{thm}

The lemma below gives sufficient conditions for $\mfP \lesW \mfQ$ to imply $\mfP \leW \mfQ$.

\begin{defn}[Dorais, Dzhafarov, Hirst, Mileti, Shafer \cite{ddhms16}]
A problem $\mfP$ is \emph{finitely tolerant} if there is a Turing functional $\Theta$ such that
whenever $p_1$ and $p_2$ are $\mfP$-instances which agree up to $t \in \N$ (i.e., for all $s \leq t$, $p_1(s) = p_2(s)$),
and $q$ is a $\mfP$-solution to $p_1$,
then $\Theta(q,t)$ is a $\mfP$-solution to $p_2$.
\end{defn}

\begin{lem}[Dzhafarov, Goh, Hirschfeldt, Patey, Pauly \cite{dghpp20}, Lemma 5.1] \label{lem:finitely_tolerant}
If $\mfP$ is finitely tolerant,
any finite modification of a $\mfP$-instance is still a $\mfP$-instance,
all $\mfP$- and $\mfQ$-solutions lie in some fixed finite set,
and $\mfP \leW \mfQ$,
then $\mfP \lesW \mfQ$.
\end{lem}

$\mflim_k$ and $\wphp{'}{m}{n}$ satisfy all conditions for both $\mfP$ and $\mfQ$ in the lemma.
In particular, $\mflim_k$ and $\wphp{'}{m}{n}$ are both finitely tolerant
as witnessed by $\Theta(q,t) = q$ for all $q \in \N^\N$, $t \in \N$.

We are ready to prove Theorem \ref{thm:reduction_lifting}, namely that the following are equivalent:
	(1) $\mflim_k \leW \wphp{'}{m}{n}$,
	(2) $\mflim_k \lesW \wphp{'}{m}{n}$,
	(3) $\mflim_k \lecsW \wphp{'}{m}{n}$,
	(4) $\id_k \lecsW \wphp{}{m}{n}$,
	(5) $\id_k \leq \wphp{}{m}{n}$,
	(6) $\id_k \lesW \wphp{}{m}{n}$.

\begin{proof}[Proof of Theorem \ref{thm:reduction_lifting}]
(1) implies (2) by Lemma \ref{lem:finitely_tolerant}.
(2) trivially implies (1).
Next, (2) trivially implies (3).
(3) implies (4) by Theorem \ref{thm:jump_inversion}.
(4) trivially implies (5).
(5) implies (6) because functions on finite discrete spaces are automatically computable.
(6) implies (2) by Proposition \ref{prop:jump_sW}.

The proof that $\wphp{}{m_0}{n_0} \leq \wphp{}{m_1}{n_1}$ if and only if $\wphp{'}{m_0}{n_0} \leW \wphp{'}{m_1}{n_1}$ is analogous.
\end{proof}	

\section{Additional results} \label{section:adhoc}

This section establishes several ad hoc reductions and nonreductions. When constructing reductions, we shall express a function from $m$ to $n$ by listing its fibers in order. For example, $(02)(13)$ represents the function $0 \mapsto 0, 2 \mapsto 0, 1 \mapsto 1, 3 \mapsto 1$. As a warm-up, we shall prove $\id_3 \leq \wphp{}{4}{2}$. Use Lemma \ref{lem:id_k_le_m_iff_k_functions_no_common_soln} and observe that the following 3 functions from $4$ to $2$ have no pairwise common solution:
\[ (01)(23) \qquad (02)(13) \qquad (03)(12). \]

Next, we address the optimality of Proposition \ref{prop:id_below_m->n_graph_decomp}. For $m = 9$ and $n = 4$, it yields $\id_4 \leq \wphp{}{9}{4}$, while for $m = 7$ and $n = 5$, it yields $\id_9 \leq \wphp{}{7}{5}$. More is true, however:

\begin{prop} \label{prop:6_leq_9->4,10_leq_7->5}
$\id_6 \leq \wphp{}{9}{4}$ and $\id_{10} \leq \wphp{}{7}{5}$.
\end{prop}
\begin{proof}
To show $\id_6 \leq \wphp{}{9}{4}$, observe that the 6 functions (from $9$ to $4$) in the first column below have no pairwise common solution. We illustrate these functions in Figure \ref{fig:6_leq_9->4}. (This is obtained from the affine plane of order 3.) 
\begin{align*}
	&(012)(38)(46)(57) &(01)(45)(2)(3)(6) \qquad &(23)(46)(0)(1)(5) \\
	&(345)(18)(26)(07) &(02)(16)(3)(4)(5) \qquad &(24)(36)(0)(1)(5) \\
	&(678)(13)(24)(05) &(03)(12)(4)(5)(6) \qquad &(25)(34)(0)(1)(6) \\
	&(036)(15)(27)(48) &(04)(13)(2)(5)(6) \qquad &(26)(35)(0)(1)(4) \\
	&(147)(56)(23)(08) &(05)(14)(2)(3)(6) \qquad & \\
	&(258)(16)(37)(04) &(06)(15)(2)(3)(4) \qquad &
\end{align*}
Similarly, consider each row in the second and third columns as a function from $7$ to $5$. These 10 functions from $7$ to $5$ have no pairwise common solution, proving $\id_{10} \leq \wphp{}{7}{5}$. (Note that $\{5,6\}$ is not a solution of any of the 10 functions. This does not affect our reduction.)
\end{proof}

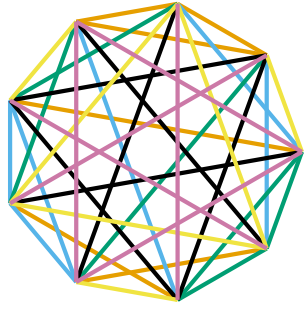
\begin{figure}
	\begin{tikzpicture}
		\newcommand{\coloredge}[3]{\draw (360*#2/9:2) edge[#1,ultra thick] (360*#3/9:2)}
		\foreach \i / \j in {1/2,1/3,2/3,4/9,5/7,6/8}
			\coloredge{wong_orange}{\i}{\j};
		\foreach \i / \j in {4/5,4/6,5/6,3/7,2/9,1/8}
			\coloredge{wong_skyblue}{\i}{\j};
		\foreach \i / \j in {7/8,7/9,8/9,1/6,2/4,3/5}
			\coloredge{wong_bluishgreen}{\i}{\j};
		\foreach \i / \j in {1/4,1/7,4/7,2/6,3/8,5/9}
			\coloredge{black}{\i}{\j};
		\foreach \i / \j in {2/5,2/8,5/8,3/4,6/7,9/1}
			\coloredge{wong_yellow}{\i}{\j};
		\foreach \i / \j in {3/6,3/9,6/9,1/5,2/7,4/8}
			\coloredge{wong_reddishpurple}{\i}{\j};
	\end{tikzpicture}
    \caption{$\id_6 \leq \wphp{}{9}{4}$. View edges of the same color as defining a function from 9 to 4.}
    \label{fig:6_leq_9->4}
\end{figure}

Next, we present additional results regarding the strictness of the hierarchies $\wphp{}{n+1}{n} \geq \wphp{}{n+2}{n} \geq \dots$ Recall from Theorem \ref{thm:summary} that $\wphp{}{3}{2} \equiv \wphp{}{4}{2} > \wphp{}{5}{2}$.

\begin{prop}
$\wphp{}{5}{2} > \wphp{}{6}{2} > \wphp{}{7}{2}$.
\end{prop}
\begin{proof}
For each $i < j < 5$, we may fix two functions $f^1_{ij}, f^2_{ij}: 5 \to 2$ whose only common solution is $\{i,j\}$, e.g., the only common solution of $(012)(34)$ and $(023)(14)$ is $\{0,2\}$. The key difference between $6$ and $5$ is that any two functions from $6$ to $2$ must have at least $2$ solutions in common.

To prove that $\wphp{}{5}{2} \not\leq \wphp{}{6}{2}$, suppose towards a contradiction that $\Phi$ and $\Psi$ witness the reduction. For each $i < j < 5$, any common solution of $\Phi(f^1_{ij}), \Phi(f^2_{ij}): 6 \to 2$ must be mapped to $\{i,j\}$ under $\Psi$. Since $\Phi(f^1_{ij})$ and $\Phi(f^2_{ij})$ have at least two solutions in common, $|\Psi^{-1}(\{i,j\})| \geq 2$. This is impossible since $\binom{6}{2} < 2\binom{5}{2}$.

Our proof that $\wphp{}{6}{2} \not\leq \wphp{}{7}{2}$ elaborates on the above ideas. First, observe that any two of the following functions from $6$ to $2$ have exactly 2 solutions in common:
\[ f_1: (034)(125) \qquad f_2: (025)(134) \qquad f_3: (024)(135) \qquad f_4: (035)(124) \]
The following functions also have exactly 2 solutions in common:
\[ f_5: (014)(235) \qquad f_6: (015)(234) \]
So we have $\binom{4}{2} + 1 = 7$ pairs of functions such that each pair has exactly 2 solutions in common. Furthermore, observe that the corresponding 7 sets of common solutions are pairwise disjoint. (Since $\binom{6}{2} < 8 \cdot 2$, we cannot find an 8th pair of functions while maintaining disjointness.)

On the other hand, a straightforward case analysis shows that any two functions from $7$ to $2$ have at least 3 solutions in common.

Suppose towards a contradiction that $\Phi$ and $\Psi$ witness $\wphp{}{7}{2} \leq \wphp{}{6}{2}$. For each of the above 7 pairs of functions $\{f,f'\}$, $\Psi$ must map each common solution of $\Phi(f)$, $\Phi(f')$ to one of the 2 common solutions of $f$ and $f'$. Since each pair $\Phi(f),\Phi(f'): 7 \to 2$ must have at least 3 solutions in common, and $\binom{7}{2} = 7 \cdot 3$, equality is forced, i.e., for each $\{f,f'\}$ above, $\Phi(f)$ and $\Phi(f')$ have exactly 3 solutions in common. This implies that $\binom{7}{2}$ is partitioned into $7$ sets of size $3$, each corresponding to the set of common solutions of some $\Phi(f_i),\Phi(f_j)$ (where $\{f_i,f_j\}$ varies among $\{f_1,f_2\},\{f_1,f_3\}, \{f_1,f_4\}, \{f_2,f_3\}, \{f_2,f_4\}, \{f_3,f_4\}$ and $\{f_5,f_6\}$).

More work is needed to reach a contradiction. Consider the common solutions of $f_5$ and $f_6$, namely $\{0,1\}$ and $\{2,3\}$. We claim that $\Psi$ maps all 3 pairs in $\Psi^{-1}\{\{0,1\},\{2,3\}\}$ to $\{2,3\}$. To prove this, first observe the following relationship between $f_7 = (023)(145)$ and the pair $\{f_1,f_2\}$: $f_1$ and $f_2$ have common solutions $\{2,5\}$ and $\{3,4\}$, neither of which is a solution to $f_7$. Therefore, $\Psi^{-1}\{\{2,5\},\{3,4\}\}$ does not contain any solution to $\Phi(f_7)$. By similar reasoning (for $\{f_3,f_4\}$), $\Psi^{-1}\{\{2,4\},\{3,5\}\}$ does not contain any solution to $\Phi(f_7)$ either.

Repeat the above reasoning for $f_8 = (045)(123)$ instead of $f_7$. We deduce that neither $\Psi^{-1}\{\{2,5\}$, $\{3,4\}\}$ nor $\Psi^{-1}\{\{2,4\},\{3,5\}\}$ contain any solution to $\Phi(f_8)$.

Next, consider the common solutions of $\{f_7,f_8\}$, namely $\{2,3\}$ and $\{4,5\}$. Since $\{0,4\}$ and $\{1,5\}$ are common solutions of $\{f_1,f_3\}$ but neither of them are common solutions of $\{f_7,f_8\}$, it follows that none of the 3 pairs in $\Psi^{-1}\{\{0,4\},\{1,5\}\}$ are common solutions of $\Phi(f_7)$ and $\Phi(f_8)$. An analogous statement holds for $\Psi^{-1}\{\{0,3\},\{1,2\}\}$, $\Psi^{-1}\{\{0,2\},\{1,3\}\}$ and $\Psi^{-1}\{\{0,5\},\{1,4\}\}$, looking at $\{f_1,f_4\}$, $\{f_2,f_3\}$ and $\{f_2,f_4\}$ respectively.

We know $\Phi(f_7)$ and $\Phi(f_8)$ have at least 3 solutions in common, and all solutions except the 3 in $\Psi^{-1}\{\{0,1\},\{2,3\}\}$ have been eliminated above. Among $\{0,1\}$ and $\{2,3\}$, only $\{2,3\}$ is a common solution of $\{f_7,f_8\}$. So $\Psi$ maps all of $\Psi^{-1}\{\{0,1\},\{2,3\}\}$ to $\{2,3\}$, proving our claim.

However, by considering $\Phi$'s output on (012)(345) and (013)(245), one may prove similarly that all 3 pairs in $\Psi^{-1}\{\{0,1\},\{2,3\}\}$ are mapped to $\{0,1\}$. We merely state the key points: None of $\{0,4\}$, $\{1,5\}$, $\{0,5\}$, $\{1,4\}$ are solutions to (012)(345) and (013)(245). (012)(345) and (013)(245) have no common solution among $\{0,3\}$, $\{1,2\}$ and $\{0,2\}$, $\{1,3\}$ and $\{2,5\}$, $\{3,4\}$ and $\{2,4\}$, $\{3,5\}$. Among $\{0,1\}$ and $\{2,3\}$, only $\{0,1\}$ is a common solution of (012)(345) or (013)(245).
\end{proof}

Recall from Theorem \ref{thm:summary} that $\wphp{}{4}{3} > \wphp{}{5}{3}$ and $\wphp{}{6}{3} > \wphp{}{7}{3}$. We fill the gap:

\begin{prop}
$\wphp{}{5}{3} > \wphp{}{6}{3}$.
\end{prop}
\begin{proof}
Suppose towards a contradiction that $\Phi$ and $\Psi$ witness $\wphp{}{5}{3} \leq \wphp{}{6}{3}$.

Case 1: Suppose $\Psi$ is surjective. Then there are at least $2\binom{5}{2} - \binom{6}{2} = 5$ pairs $i \neq j$ from $\binom{5}{2}$ such that $|\Psi^{-1}(\{i,j\})| \leq 1$. Among them, choose any two pairs $\{a,b\}$ and $\{c,d\}$ such that $a\neq c,d$ and $b \neq c,d$. Define $f: 5 \to 3$ by $f(a) = f(b) = 0$, $f(c) = f(d) = 1$, otherwise $f(e) = 2$. The only solutions of $f$ are $\{a,b\}$ and $\{c,d\}$. Since $\Phi(f)$ is a function from $6$ to $3$, it has at least 3 solutions, which (because $|\Psi^{-1}(\{a,b\})| \leq 1$ and $|\Psi^{-1}(\{c,d\})| \leq 1$) means that at least one of its solutions cannot map to $\{a,b\}$ or $\{c,d\}$ under $\Psi$. This means that $\Phi$ and $\Psi$ do not witness the reduction after all.

Case 2: Otherwise, fix $\{a,b\}$ not in the range of $\Psi$. Let $c$, $d$ and $e$ denote the remaining numbers in 5. Then $\Phi((ab)(cd)): 6 \to 3$ has at least 3 solutions, all of which must map to $\{c,d\}$ under $\Psi$. This means $|\Psi^{-1}(\{c,d\})| \geq 3$. Similarly, $|\Psi^{-1}(\{c,e\})| \geq 3$ and $|\Psi^{-1}(\{d,e\})| \geq 3$ by considering $\Phi((ab)(ce))$ and $\Phi((ab)(de))$ respectively.

Next, observe that each of $\Phi((ac)(bd))$, $\Phi((ad)(be))$ and $\Phi((ae)(bc))$ have at least 3 solutions, and that they have no pairwise common solution (because $(ac)(bd)$, $(ad)(be)$, and $(ae)(bc)$ have no pairwise common solution).

But now we have $6$ pairwise disjoint subsets of $\binom{6}{2}$ each with size at least $3$: $\Psi^{-1}(\{c,d\})$, $\Psi^{-1}(\{c,e\})$, $\Psi^{-1}(\{d,e\})$, $\Psi^{-1}(\{\{a,c\},\{b,d\}\})$, $\Psi^{-1}(\{\{a,d\},\{b,e\}\})$, and $\Psi^{-1}(\{\{a,e\},\{b,c\}\})$. This is impossible.
\end{proof}

Finally, we construct a reduction from $\C_3$ to $\wphp{'}{8}{2}$. In addition to strengthening Corollary \ref{cor:AoUC_3_below_n^3->n} for $n = 2$, this yields a complete picture for when $\C_k \leW \wphp{'}{m}{2}$:
if $m \geq 9$,
then $\C_2 = \ACC_2 \not\leW \wphp{'}{m}{2}$ by Proposition \ref{prop:ACC_not_below_n^(k+1)+1};
if $3 \leq m \leq 8$,
then $\C_3 \leW \wphp{'}{m}{2}$ by the following proposition
and the fact that $\C_4 \not\lecW \wphp{'}{m}{2}$.
(To see this last fact, note that $\wphp{'}{m}{2} \leW \wphp{'}{3}{2} \leW \mflim_3$ by Theorem \ref{thm:reduction_lifting}. The claim now follows from the more general fact that $\C_{k+1} \not\lecW \mflim_k$ for any $k$, which is standard, but we sketch here for completeness. Indeed, suppose otherwise, as witnessed by continuous functionals $\Phi$ and $\Psi$. It can be checked that there exist numbers $s_0,\ldots,s_{k-1}$, $i_0,\ldots,i_{k-1}$, and $j_0,\ldots,j_k$ with the following properties: $i_0,\ldots,i_{k-1}$ are all $< k$ and distinct; $j_0,\ldots,j_k$ are all $< k+1$ and distinct; and for every $\ell < k$, if $p \in \N^\N$ is any sequence extending $j_0^{s_0}\cdots j_\ell^{s_\ell}$ then $\Phi(p,i_\ell) = j_\ell$. Let $p = j_0^{s_0} \cdots j_{k-1}^{s_{k-1}} (k+1)^\N$, an instance of $\C_{k+1}$ with unique solution $j_k$. Then $\Phi(p)$ is an instance of $\lim_k$, which must thus have solution $i_\ell$ for some $\ell < k$. But $\Psi(p,i_\ell) = j_\ell \neq j_k$, a contradiction.)

\begin{prop} \label{prop:C_3_leq_8->2}
$\C_3 \leW \wphp{'}{8}{2}$.
\end{prop}

\begin{proof}
We shall use Theorem \ref{thm:C_k_leW_m_iff_functions_indexed_by_injections}. Consider the following functions from $8$ to $2$:
\begin{align*}
f_\emptyset &= (0123)(4567) \\
f_0 &= (0145)(2367) \\
f_{01} = f_{10} &= (0167)(2345) \\
f_2 = f_{02} &= (0347)(1256) \\
f_1 = f_{21} &= (0246)(1357) \\
f_{12} &= (0257)(1346) \\
f_{20} &= (0356)(1247)
\end{align*}
It is perhaps clearer to visualize them in a tree,
where each node corresponds to an injection from a proper initial segment of $3$ to $3$,
obtained by reading the labels above it from top to bottom:
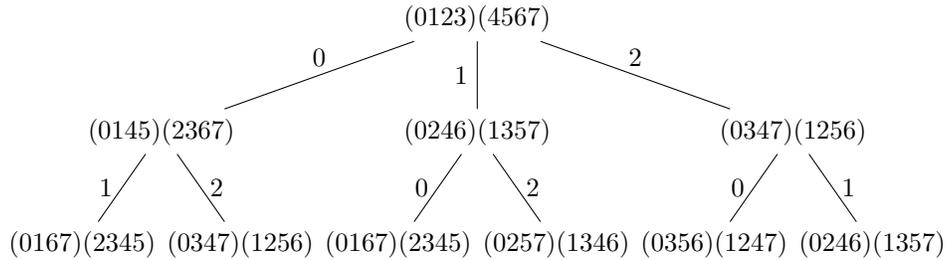
\begin{center}
\begin{tikzpicture}
	[level distance=15mm,level/.style={sibling distance=42mm/#1}]
	\node {(0123)(4567)}
		child {node {(0145)(2367)}
			child {node {(0167)(2345)} edge from parent node[left] {1}}
			child {node {(0347)(1256)} edge from parent node[right] {2}}
		edge from parent node[above] {0}}
		child {node {(0246)(1357)}
			child {node {(0167)(2345)} edge from parent node[left] {0}}
			child {node {(0257)(1346)} edge from parent node[right] {2}}
		edge from parent node[left] {1}}
		child {node {(0347)(1256)}
			child {node {(0356)(1247)} edge from parent node[left] {0}}
			child {node {(0246)(1357)} edge from parent node[right] {1}}
		edge from parent node[above] {2}};
\end{tikzpicture}
\end{center}
One may check that for each of the $\binom{8}{2} = 28$ possible $\{i,j\}$ and each of the $f_\sigma$ with $\{i,j\}$ as a solution,
there is some $\ell < 3$ such that for all $\tau$ extending $\sigma$ with $\{i,j\}$ as a solution for $f_\tau$,
$\ell$ is not in the range of $\tau$. For example, $\{0,6\}$ is a solution of $f_{01}$, $f_1$, $f_{10}$, $f_{20}$, and $f_{21}$.
The only interesting case is $f_1$ (since $10$ extends $1$),
but since $\{0,6\}$ is \emph{not} a solution for $f_{12}$,
we can choose $\ell = 2$ in this case.
\end{proof}

\bibliographystyle{plain}
\bibliography{reductions_php}

\end{document}